\documentclass[11pt]{amsart}
\usepackage{amsmath,amsthm,amssymb, amscd, amsfonts}

\theoremstyle{plain}
\newtheorem{theorem}{Theorem}[section]
\newtheorem{corollary}[theorem]{Corollary}

\newtheorem{proposition}[theorem]{Proposition}
\newtheorem{lemma}[theorem]{Lemma}

\theoremstyle{definition}

\newtheorem{example}[theorem]{Example}

\theoremstyle{remark}
\newtheorem{remark}[theorem]{Remark}

\numberwithin{equation}{section}\theoremstyle{plain}

\newcommand{\Z}{{\mathbb Z}}
\newcommand{\B}{{\mathcal B}}
\newcommand{\A}{{\mathcal A}}
\newcommand{\J}{{\mathcal J}}
\newcommand{\C}{{\mathbb C}}
\newcommand{\Cc}{{\mathcal C}}
\newcommand{\Oo}{{\mathcal O}}

\newcommand\id{\operatorname{id}}

\newcommand\vect{\operatorname{Vec}}
\newcommand\ad{\operatorname{ad}}
\newcommand\coad{\operatorname{coad}}
\newcommand\Tr{\operatorname{Tr}}

\newcommand\co{\operatorname{co}}

\newcommand\comod{\operatorname{comod}}

\newcommand\End{\operatorname{End}}

\newcommand\Aut{\operatorname{Aut}}
\newcommand\Opext{\operatorname{Opext}}
\newcommand\res{\operatorname{res}}

\newcommand\Ind{\operatorname{Ind}}
\newcommand\SL{\operatorname{SL}}
\newcommand\PSL{\operatorname{PSL}}
\newcommand\GL{\operatorname{GL}}
\newcommand\SU{\operatorname{SU}}
\newcommand\SO{\operatorname{SO}}

\newcommand\Hom{\operatorname{Hom}}
\newcommand\Rep{\operatorname{Rep}}

\begin{document}
\title[Deformations of binary polyhedral groups]{Hopf algebra deformations of binary polyhedral groups}
\author{Julien Bichon}
\address{Julien Bichon:
Laboratoire de Math\'ematiques,
Universit\'e Blaise Pascal,
Complexe universitaire des C\'ezeaux,
63177~Aubi\`ere Cedex, France}
\email{Julien.Bichon@math.univ-bpclermont.fr
\newline \indent \emph{URL:}\/ http://www.math.univ-bpclermont.fr/-bichon}
\author{Sonia Natale}
\address{Sonia Natale: Facultad de Matem\'atica, Astronom\'\i a y F\'\i sica.
Universidad Nacional de C\'ordoba. CIEM -- CONICET. (5000) Ciudad
Universitaria. C\'ordoba, Argentina} \email{natale@famaf.unc.edu.ar
\newline \indent \emph{URL:}\/ http://www.famaf.unc.edu.ar/$\sim$natale}

\thanks{This work was partially supported by CONICET--CNRS, ANPCyT, SeCYT (UNC), FaMAF, Alexander von Humboldt Foundation, and the ANR project Galoisint.}


\date{May 21, 2010.}

\begin{abstract} We show that semisimple Hopf algebras having
a self-dual faithful irreducible comodule of dimension $2$ are
always obtained as abelian extensions with quotient $\Z_2$. We
prove that nontrivial Hopf algebras arising in this way can be
regarded as deformations of binary polyhedral groups and describe
its category of representations. We also prove a strengthening of
a result of Nichols and Richmond on cosemisimple Hopf algebras
with a $2$-dimensional irreducible comodule in the finite
dimensional context. Finally, we give some applications to the
classification of certain classes of semisimple Hopf algebras.
\end{abstract}

\maketitle

\section{Introduction and main results}

Throughout the paper we work over  $k$,  an algebraically closed field of characteristic zero.
This paper is a contribution to the study of low dimensional semisimple
Hopf algebras, namely those that admit a faithful comodule of dimension $2$.
Our starting point is a theorem by Nichols and Richmond \cite{NR}, that
we recall first.

Let $H$ be a cosemisimple Hopf algebra over $k$ and let
$C\subseteq H$ be a simple subcoalgebra of dimension $4$ (this is equivalent to say that
$H$ has
an irreducible comodule of dimension $2$). The main
result of the paper \cite{NR} says that, in this case, at least
one of the following possibilities holds:
\begin{enumerate}\item\label{1} $H$ has a group-like element $g$ of order $2$
such that $gC = C$;
\item\label{2} $H$ has a Hopf subalgebra of dimension $24$, which contains
a group-like element $g$ of order $2$ such that $gC \neq C$;
\item\label{3} $H$ has a Hopf subalgebra of dimension $12$ or $60$;
\item\label{4} $H$ has a family $\{C_n: \, n\geq 1\}$ of simple
subcoalgebras such that $\dim C_n = n^2$, and for all $n\geq 2$,
\begin{equation*}C_n\mathcal S(C_2)  = C_{n-1} + C_{n+1}, \; n \text{
even,}\qquad  C_nC_2  = C_{n-1} + C_{n+1}, \; n \text{
odd.}\end{equation*}
\end{enumerate} See \cite[Theorem 11]{NR}.
Note that if $H$ is finite dimensional only the possibilities
\eqref{1}--\eqref{3} can arise, and then $\dim H$ is even (in particular this gives
a partial positive answer to one of Kaplansky's long-standing conjectures).

\medbreak In this paper we obtain a more precise description of
this result, in case $H$ is finite dimensional. We prove in
Subsection \ref{subgrupos} the following theorem that strengthens
the Nichols-Richmond Theorem above in the finite dimensional
context, and connects this class of Hopf algebras with polyhedral
groups.

\begin{theorem}\label{nr-ref} Let $H$ be a cosemisimple finite dimensional Hopf
algebra. Suppose $H$ contains a simple subcoalgebra $C$ of
dimension $4$. Then the subalgebra $B = k[C\mathcal S(C)]$ is a
commutative Hopf subalgebra of $H$ isomorphic to $k^{\Gamma}$,
where $\Gamma$ is a non cyclic finite subgroup of $\PSL_2(k)$ of even order.

\medbreak Furthermore, let $\chi \in C$ be the irreducible
character contained in $C$, and let $G[\chi] \subseteq G(H)$ be
the stabilizer of $\chi$ with respect to left multiplication. Then
$|G[\chi]|$ divides $4$, and
the following hold:
\begin{enumerate}\item[(i)] If $|G[\chi]| = 4$, then $B \simeq k^{\Z_2\times \Z_2}$.
\item[(ii)] If $|G[\chi]| = 2$, then $B \simeq k^{D_n}$, where $n \geq 3$.
\item[(iii)] If $|G[\chi]| = 1$, then $B \simeq k^{\mathbb A_4}$,
$k^{\mathbb S_4}$, or $k^{\mathbb A_5}$. \end{enumerate}
\end{theorem}

Observe that, in addition, the category of $B$-comodules is contained in
the adjoint subcategory of the fusion category $H-\comod$ of
finite dimensional $H$-comodules.

\medbreak In special situations, the above theorem may be refined
as follows. Indeed, assume that the irreducible $2$-dimensional
comodule $V$ corresponding to $C$ satisfies $V^* \otimes V \simeq
V \otimes V^*$. Then the Hopf subalgebra $H'$ of $H$ generated by
$C$ fits into a cocentral exact sequence
$$k\to k^\Gamma \to H' \to  k\Z_m\to k$$
for  a polyedral group $\Gamma$ of even order and $m\geq 1$.
See Theorem \ref{shortspecial}.

Moreover when $V$ is self-dual, the Hopf algebra $H'=k[C]$ generated by $C$ can be
completely described. See Theorem \ref{clasificacion} below.

\medbreak Compact quantum subgroups of $\text{SU}_{-1}(2)$, or in
other words, quotient Hopf $*$-algebras $\Oo_{-1}[\text{SU}_2(\C)]
\to H$, were determined in \cite{podles}. In this paper we give an
algebraic version of this result. We obtain a description, using
Hopf algebra exact sequences, of quotient Hopf algebras
$\Oo_{-1}[\SL_2(k)] \to H$ in the finite dimensional case. See
Theorem \ref{explicito}. Our approach gives a more precise
description of such quotients, and also implies a description of
their tensor categories of representations in terms of
$\Z_2$-equivariantizations of certain pointed fusion categories.
See Corollary \ref{representations}.

It turns out that every such Hopf algebra $H$ can be regarded as a
deformation of an appropriate binary polyhedral group. We remark
that such groups admit no nontrivial cocycle deformation in the
sense of \cite{movshev, davydov}, since every Sylow subgroup is
either generalized quaternion or cyclic.

\medbreak We combine this result with the classification of Hopf
algebras associated to non-degenerate bilinear forms in the
$2\times 2$ case \cite{B, DL}, to prove the following theorem.

\begin{theorem}\label{clasificacion} Let $H$ be a semisimple Hopf
algebra over $k$. Assume $H$ has a  self-dual faithful irreducible
comodule $V$ of dimension $2$. Let $\nu(V) = \pm 1$ denote the
Frobenius-Schur indicator of $V$.

Then we have:

\begin{enumerate}\item[(i)] Suppose $\nu(V) = -1$. Then  $H$ is commutative. Moreover,
$H \simeq k^{\widetilde \Gamma}$,  where $\widetilde \Gamma$ is a
non-abelian binary polyhedral group.

\item[(ii)] Suppose $\nu(V) = 1$. Then either $H$ is commutative and
isomorphic to $k^{D_n}$, $n \geq 3$, or $H$ is isomorphic to one
of the nontrivial Hopf algebra deformations $\mathcal A[\widetilde
\Gamma]$ or $\mathcal B[\widetilde \Gamma]$ of the binary
polyhedral group $\widetilde \Gamma$, as in Subsection
\ref{descripcion}. \end{enumerate}\end{theorem}

Let $C \subseteq H$ be the simple subcoalgebra containing $V$.
Note that the assumptions on $V$ in Theorem \ref{clasificacion}
amount to the assumption that $\dim C = 4$, $\mathcal S(C) = C$
and $C$ generates $H$ as an algebra.

\medbreak Theorem \ref{clasificacion} is proved in Subsection
\ref{descripcion}. We point out that only one nontrivial
deformation, $\mathcal B[\widetilde I]$, can occur in the case of
the binary icosahedral group $\widetilde I$. The algebra and
coalgebra structures of the Hopf algebras $\mathcal A[\widetilde
\Gamma]$, $\mathcal B[\widetilde \Gamma]$, are described in Remark
\ref{chardeg}. Their fusion rules are the same as the ones of
$\widetilde{\Gamma}$, see Remark \ref{fusion}.

In particular, if $H$ is one of the Hopf algebras in (ii), then
$\dim H = |\widetilde \Gamma|$ can be either  $24$, $48$, $120$ or
$4n$, $n \geq 2$. It turns out that, when $k = \C$ is the field of
complex numbers, such Hopf algebras admit a Kac algebra structure.

Furthermore, let $\vect^{\Gamma}$ be the fusion category of
$\Gamma$-graded vector spaces, where $\Gamma \simeq \widetilde
\Gamma /Z(\widetilde \Gamma)$ is the corresponding polyhedral
group. Then we have an equivalence of fusion categories $\Rep H
\simeq (\vect^{\Gamma})^{\Z_2}$, where the last category is the
equivariantization of the category $\vect^{\Gamma}$ with respect
to an appropriate action of $\Z_2$.

\medbreak The Hopf algebras $\mathcal A[\widetilde D_n]$,
$\mathcal B[\widetilde D_n]$, $n \geq 3$, were constructed and
studied by Masuoka in \cite{ma-contemp}.  By the results in
\textit{loc. cit.}, $\mathcal A[\widetilde D_n]$ is the only
nontrivial cocycle deformation of the commutative Hopf algebra
$k^{D_{2n}}$. On the other hand, the Hopf algebra $\mathcal
B[\widetilde D_n]$ is self-dual and it has no nontrivial cocycle
deformations.

In the case when $n = 2$, $\mathcal A[\widetilde D_2] \simeq
\mathcal B[\widetilde D_2]$, and they are both isomorphic to the
$8$-dimensional Kac-Paljutkin Hopf algebra $H_8$.

\medbreak We give some applications of the main results to the
classification of a special class of semisimple Hopf algebras. We
consider a semisimple  Hopf algebra $H$ such that $\deg \chi \leq
2$, for all irreducible character $\chi \in H$.

We show in Theorem \ref{1m2n} that such a Hopf algebra $H$ is not
simple. Moreover, we prove the existence of certain exact
sequences for $H$, which imply that either $H$ or $H^*$ contains a
nontrivial central group-like element and, in addition, that $H$
is lower semisolvable. See Corollary \ref{semisoluble}.

An application to semisimple Hopf algebras of dimension $60$ is
also given in Subsection \ref{dim60}. We show that for two known
examples of simple Hopf algebras of this dimension, the Hopf
algebras are determined by its coalgebra types.

\medbreak The paper is organized as follows: in Section \ref{one}
we recall some facts on Hopf algebra extensions and their
relations with the categories of representations. Section
\ref{nr-th} is devoted to the proof of Theorem \ref{nr-ref}. In
Section \ref{two} we recall the definition and basic properties of
the Hopf algebra of a non-degenerate bilinear form and its
relation with the classification problem we consider. We study in
this section a canonical $\Z_2$-grading in the category of finite
dimensional comodules. In Section \ref{2por2} we discuss the
$2\times 2$ case and further facts related to the Hopf algebra
$\Oo_{-1}[\SL_2(k)]$ and the classification of its finite
dimensional quotients. Section \ref{appl} contains some
applications to the classification of certain classes of
semisimple Hopf algebras. We include at the end an Appendix with
some facts concerning a special kind of $\Opext$ groups that
include those related to the Hopf algebras $\mathcal A[\widetilde
\Gamma]$, $\mathcal B[\widetilde \Gamma]$.

\subsection*{Acknowledgement} The authors thank Florian Steiner for pointing out a mistake in a previous version of the proof of Proposition 3.2. The research of S. Natale was done during a stay at the Mathematisches Institut der Universit\" at M\" unchen, as an Alexander von Humboldt Fellow. She thanks Prof. Hans-J\" urgen Schneider for the kind hospitality.
\section{Preliminaries}\label{one}

\subsection{Hopf algebra exact sequences}\label{exact-seq}

Recall that a sequence  of Hopf algebra maps
\begin{equation}\label{seq}k \to K \overset{i}\to H \overset{p}\to \overline H \to
k,\end{equation} is called \emph{exact} if the following
conditions hold:
\begin{enumerate}\item $i$ is injective and $p$ is surjective,
\item $p \circ i$ = $\epsilon 1$,
\item $\ker p =HK^+$, \item $K = H^{\co p} = \{ h \in H:\, (\id \otimes p)\Delta(h) = h \otimes 1
\}$. \end{enumerate} Assume we have a sequence \eqref{seq}. If $H$
is faithfully flat over $K$ (in particular if $H$ is finite
dimensional), then (1), (2) and (3) imply (4). On the other hand,
if $H$ is faithfully coflat over $\overline H$, then (1), (2) and
(4) imply (3). See e.g. \cite{ad}.

\medbreak We shall say that the Hopf algebra map $p: H \to
\overline H$ is \emph{cocentral} if $p(h_1) \otimes h_2 = p(h_2)
\otimes h_1$, for all $h\in H$.

\begin{lemma}\label{exactness} Let $H$ be a Hopf algebra. Let also $G$ be a finite group and let $p: H \to kG$ be a
surjective Hopf algebra map. Suppose $p$ is cocentral. Then there
is an exact sequence of Hopf algebras $k \to K \to H
\overset{p}\to kG \to k$, where $K = H^{\co p}$. \end{lemma}

\begin{proof} Let $I = \ker p$. Then $I$ is a coideal of $H$ of
finite codimension. Since $p$ is cocentral by assumption, then $I$
is left and right conormal.  Also, $K = H^{\co p} = {}^{\co p}H$
is a Hopf subalgebra of $H$, and thus we have a sequence of Hopf
algebra maps \eqref{seq}, that satisfies (1), (2) and (4). By
\cite[2.1]{schneider}, $H$ is left and right faithfully coflat
over $H/I \simeq kG$. Hence the sequence is exact, as claimed.
\end{proof}

A Hopf subalgebra $K$ of $H$ is called \emph{normal} if it is
stable under both left and right adjoint actions of $H$ on itself,
defined, respectively, by
$$\ad_l(h)(x) = h_1x\mathcal S(h_2), \quad \ad_r(h)(x) = \mathcal S(h_1)x
h_2,$$ for all $h, x \in H$. If $K \subseteq H$ is a normal Hopf
subalgebra, then the ideal $HK^+$ is a Hopf ideal and the
canonical map $H \to \overline H : = H/HK^+$ is a Hopf algebra
map. Hence, if $H$ is faithfully flat over $K$,  there is an exact
sequence \eqref{seq} where all maps are canonical. A Hopf algebra
is said to be simple if it has no proper normal Hopf subalgebra.

\medbreak Suppose, from now on, that $H$ is finite dimensional.
Then the exact sequence above is cleft and $H$ is isomorphic to a
bicrossed product $H \simeq K\#_{\sigma}^{\tau}\overline H$, for
appropriate compatible data.

\medbreak The exact sequence \eqref{seq} is called \emph{abelian}
if $K$ is commutative and $\overline H$ is cocommutative. In this
case $K \simeq k^N$ and $\overline H \simeq k F$, for some finite
groups $N$ and $F$; in particular, $H$ is semisimple.

As explained in \cite[Section 4]{Maext} such an extension is
determined by a matched pair $(F, N)$ with respect to compatible
actions $\vartriangleright: N \times F \to F$ and
$\vartriangleleft: N \times F \to N$, and $2$-cocycles $\sigma: k
F\otimes k F \to (k^{N})^{\times}$ and $\tau: kN\otimes kN \to
(k^F)^{\times}$, subject to appropriate compatibility conditions.
For a fixed matched pair $(F, N)$, the isomorphism classes of
extensions $k \to k^N \to H \to kF \to k$, or equivalently, the
equivalence classes of pairs $(\sigma, \tau)$, form an abelian
group that we shall denote $\Opext(k^N, kF)$.

\medbreak Since they will appear later on in this paper, we recall
here the explicit formulas for the multiplication,
comultiplication and antipode in the bicrossed product $k^N
{}^{\tau}\#_{\sigma}kF$ representing the class of $(\sigma, \tau)$
in $\Opext(k^N, kF)$:
\begin{align}\label{mult}(e_s \# x)  (e_t \# y)
& =  \delta_{s \vartriangleleft x, t} \, \sigma_s(x, y)  \, e_s \# xy, \\
\label{comult}\Delta(e_s \# x) & = \sum_{gh = s} \tau_x(g, h) \,
e_g \# (h \vartriangleright x) \otimes e_h \# x,\end{align}
\begin{equation}\label{antipode} \mathcal S(e_g \# x)  = \sigma_{(g
\vartriangleleft x)^{-1}}((g\vartriangleright x)^{-1},
g\vartriangleright x)^{-1} \, \tau_x(g^{-1}, g)^{-1} \,
e_{(g\vartriangleleft x)^{-1}} \# (g\vartriangleright x)^{-1},
\end{equation}
for all $s, t \in N$, $x, y \in F$, where $\sigma_s(x, y) =
\sigma(x, y) (s)$ and $\tau_x(s, t) = \tau(s, t) (x)$.

\begin{remark}\label{coalg-abel} Suppose that the action $\vartriangleright:N \times F \to
F$ is trivial. Then it follows from formula \eqref{comult} that
the subspaces $R_x : = k^N\#x$ are subcoalgebras of $H =
k^N{}^{\tau}\#_{\sigma}kF$ and we have a decomposition of
coalgebras $H = \oplus_{x\in F}R_x$.

Moreover, in this case, $\tau_x: N\times N \to k^{\times}$ is a
$2$-cocycle on $N$ and, as coalgebras, $k^N \# x \simeq
(k_{\tau_x}N)^*$, where $k_{\tau_x}N$ denotes the twisted group
algebra. \end{remark}

\subsection{Type of a cosemisimple Hopf algebra}

For a Hopf algebra $H$ we shall denote by $H$-comod the tensor
category of its finite dimensional (right) comodules.

Let $H$ be a finite dimensional cosemisimple Hopf algebra over
$k$. As a coalgebra, $H$ is isomorphic to a direct sum of full
matrix coalgebras
\begin{equation}\label{estructura} H \simeq k^{(n)} \oplus \bigoplus_{d_i > 1}
M_{d_i}(k)^{(n_i)},\end{equation} where $n = |G(H)|$.

If we have an isomorphism as in \eqref{estructura},  we shall say
that $H$ is \emph{of type} $(1, n; d_1, n_1; \dots; d_r, n_r)$
\emph{as a coalgebra}. If $H^*$, which is also cosemisimple, is of
type $(1, n; d_1, n_1; \dots )$ as a coalgebra, we shall say that
$H$ is \emph{of type} $(1, n; d_1, n_1; \dots )$ \emph{as an
algebra}.

So that $H$ is of type $(1, n; d_1, n_1; \dots; d_r, n_r)$ as a
(co-)algebra if and only if $H$ has $n$ non-isomorphic
one-dimensional (co-)representations,  $n_1$ non-isomorphic
irreducible (co-)representations of dimension $d_1$, etc.

\medbreak By an \emph{irreducible character} $\chi \in H$, we
shall mean the character $\chi = \chi_V$ of an irreducible
corepresentation $V$ of $H$. Hence, $\chi_V$ is a cocommutative
element of $H$, such that $f(\chi) = \Tr_V(f)$, for all $f \in
H^*$, regarding $V$ as an irreducible representation of $H^*$. The
dimension of $V$ will be called the \emph{degree} of $\chi$,
denoted $\deg \chi$.

Every irreducible character $\chi \in H$ is contained in a unique
simple subcoalgebra. If this subcoalgebra is isomorphic to
$M_{d_i}(k)$, then $d_i = \deg \chi$.

\medbreak Suppose $H$ is of type $(1, n; d_1, n_1; \dots; d_r,
n_r)$ as a coalgebra. It follows from the Nichols-Zoeller Theorem
\cite{NZ}, that $n$ divides both $\dim H$ and $n_i d_i^2$, for all
$i$. Moreover, if $d_i = 2$ for some $i$, then the dimension of
$H$ is even \cite{NR}. Also, by \cite{ZS}, if $n = 1$, then $\{
d_i : d_i > 1 \}$ has at least three elements.

\subsection{The adjoint subcategory}\label{adjoint}

Let  $\Cc$ be a $k$-linear  abelian semisimple tensor category. Recall from \cite[8.5.]{ENO} that the \emph{adjoint
subcategory} $\Cc_{\ad}$ is the full tensor subcategory generated by the
simple components of $V\otimes V^*$, where $V$ runs over the set
of (isomorphism classes of) simple objects of $\Cc$.

\medbreak Suppose $\Cc = \Rep H$ is the fusion category of finite
dimensional representations of $H$, where $H$ is a semisimple
(hence finite dimensional) Hopf algebra over $k$.

In this case there exists a unique maximal \emph{central} Hopf
subalgebra $K$ of $H$, such that $\Cc_{\ad} = \Rep H_{\ad}$, where
$H_{\ad} = H/HK^+$ is the corresponding quotient Hopf algebra.
Moreover, one has $K = k^{U(\Cc)}$, where $U(\Cc)$ is the
\emph{universal grading group} of $\Cc$. See \cite[Theorem
3.8]{gelaki-nik}. In particular, there is a central extension of
Hopf algebras $$k \to k^{U(\Cc)} \to H \to H_{\ad} \to k.$$

There is also a dual construction associated to
the category $H-\comod$, for $H$ cosemisimple.
Indeed, the Hopf subalgebra $H_{\coad}$ of $H$ determined
by $H_{\coad}-\comod = (H-\comod)_{\ad}$
gives a universal cocentral exact sequence
\begin{equation}\label{univcocentral}k \to H_{\coad} \to H \to kU(\Cc')\to k
\end{equation}
where  $\Cc' = H-\comod$.

\medbreak Note that the notion of simplicity of a finite dimensional Hopf
algebra is self-dual. As a consequence of the above, we get the
following lemma:

\begin{lemma}\label{h-ad} Suppose $H$ is simple and not commutative. Then $H_{\ad} = H=H_{\coad}$. \qed \end{lemma}

In other words, if $H$ is simple, then $H$ is generated as algebra
by the simple subcoalgebras appearing in the decomposition of
$C\mathcal S(C)$, where $C$ runs over the simple subcoalgebras of
$H$.

\begin{lemma} \label{chichi*} Suppose $H$ contains a faithful irreducible character
$\chi$. Assume in addition that $\chi \chi^* = \chi^* \chi$. Then
$H_{\coad} = k[C\mathcal S(C)]$, where $C \subseteq H$ is the
simple subcoalgebra containing $\chi$. \end{lemma}

\begin{proof} We have $k[C\mathcal S(C)] \subseteq H_{\coad}$.
The faithfulness assumption on $\chi$ means that $H = k[C]$ and
for every irreducible character $\lambda$ of $H$, using the
assumption $\chi \chi^* = \chi^* \chi$, there exist nonnegative
integers $m$ and $p$ such that $\lambda$ appears with positive
multiplicity in $\chi^m(\chi^*)^p$.
Hence $\lambda^*$
appears with positive multiplicity in $\chi^p(\chi^*)^m$ (since
$H$ is cosemisimple, we have $\chi^{**}=\chi$). Thus
$\lambda\lambda^*$ appears with positive multiplicity in
$\chi^{m+p}(\chi^*)^{m+p}$. Then $\lambda \lambda^* \in
k[C\mathcal S(C)]$. Since this holds for any irreducible character
$\lambda$, we get $H_{\coad} \subseteq k[C\mathcal S(C)]$ and the
lemma follows.
\end{proof}

\section{On a theorem of Nichols and Richmond}\label{nr-th}

In this section we aim to give a proof of our most general  result on
finite dimensional cosemisimple Hopf algebras with a simple
comodule of dimension $2$, strengthening, in this context, the
Nichols-Richmond Theorem \cite{NR}.
The proof is based on the algebraic generalization of some
well known results in the theory of compact quantum groups
\cite{ba1,ba2}.

\subsection{Comodule algebras of dimension $4$} Let $V$ be a finite dimensional $H$-comodule. Then
$\End V \simeq V \otimes V^*$ is an $H$-comodule algebra with the
diagonal coaction of $H$ and multiplication given by composition.

\begin{remark}\label{traza} Under the canonical ($H$-colinear) isomorphism $V \otimes V^*
\simeq \End V$, the trace map $\Tr: \End V \to k$ becomes
identified with the evaluation map $V \otimes V^* \to k$, $v
\otimes f \to f(v)$. If $H$ is involutory, that is, if $\mathcal
S^2 = \id$, then $\Tr: \End V \to k$ is an $H$-comodule map.

\medbreak Suppose $H$ is cosemisimple and $\mathcal S^2 = \id$.
Let $V \ncong k$ be a simple (hence finite dimensional)
$H$-comodule, and decompose $V \otimes V^* = k \oplus W$, where $W
\subseteq V$ is an $H$-subcomodule. Then we have $W = \ker \Tr$.

Indeed, the restriction  $\Tr\vert_{W}$ is an $H$-comodule map.
Therefore $\Tr\vert_{W} = 0$, since $\Hom^H(W, k) = 0$, $V \ncong
k$ being simple. \end{remark}

The following result is related to the computation
of the compact quantum symmetry group of the $2 \times 2$ complex matrix
algebra given in \cite{ba2}.

\begin{proposition}\label{coefA} Let $H$ be a Hopf algebra and let
$A$ be an $H$-comodule algebra such that $A \simeq M_2(k)$. Assume
that the trace map $\Tr: A \to k$ is an $H$-comodule map. Then the
subalgebra $k[C_A]$ generated by the subcoalgebra $C_A$ of matrix
coefficients of $A$ is a commutative Hopf subalgebra of $H$.

We have moreover $k[C_A] \simeq \Oo[\Gamma]$, where $\Gamma$ is an
algebraic subgroup of $\PSL_2(k)$. \end{proposition}

\begin{proof} Let $\rho: A \to A \otimes H$ be the $H$-comodule
structure on $A$. Consider the basis of $A$ consisting of the unit
quaternions
\begin{equation*} e_0 = \begin{pmatrix}1 & 0\\0 & 1\end{pmatrix},
e_1 = \begin{pmatrix}-\sqrt{-1} & 0\\0 & \sqrt{-1}\end{pmatrix},
e_2 = \begin{pmatrix}0 & -1\\1 & 0\end{pmatrix}, e_3 =
\begin{pmatrix}0 & \sqrt{-1}\\ \sqrt{-1} & 0\end{pmatrix}.
\end{equation*} These satisfy the following multiplication rules:
\begin{equation}\label{quater}e_i^2 = -e_0, \; 1\leq i \leq 3, \quad e_1e_2 = e_3, \quad e_2e_3 = e_1,
\quad e_3e_1 = e_2. \end{equation} Moreover, $\{ e_1, e_2, e_3 \}$
forms a basis of $W$, the kernel of the trace map $\Tr: M_2(k) \to k$.
Since by assumption, this is an $H$-subcomodule of $A$, then we
can write $\rho(e_0) = e_0 \otimes 1$ and  $$\rho(e_i) =
\sum_{j=1}^3 e_j \otimes x_{ji},$$ for all $1\leq i \leq 3$, where
$x_{ij} \in H$. In particular, $x_{ij}$, $1 \leq i, j \leq 3$,
generate the subalgebra $k[C_A]$. We have moreover,
\begin{equation}\label{delta}\Delta(x_{ij}) = \sum_lx_{il} \otimes x_{lj}, \quad 1\leq i,
j \leq 3.\end{equation}
The map $W \otimes W \to k$, $e_i\otimes e_j \mapsto \delta_{ij}$, is $H$-colinear  since
it is the composition of $H$-colinear maps
\begin{equation*}\begin{CD}
W \otimes W \subset M_2(k) \otimes M_2(k) @>{\rm mult}>> M_2(k)
@>-\frac{1}{2} {\rm Tr}>> k.
\end{CD}
\end{equation*}
Hence we see that the matrix $(x_{ij})$ is orthogonal, and we have $S(x_{ij})=x_{ji}$, $1\leq i,
j \leq 3$ (see Section 4 for details).

Since $\rho$ is an algebra map, it must preserve the relations
\eqref{quater}. The relations $\rho(e_i)^2 = - e_0 \otimes 1$,
$1\leq i \leq 3$, amount to
$$\label{rels1} x_{li}x_{ti}  =
x_{ti}x_{li}$$ for all $i = 1, 2, 3$. This
means that the matrix coefficients in a given
column commute pairwise. Using the antipode, we get
$$\label{rels2} x_{li}x_{lj} =x_{lj}x_{li}$$
and thus the  matrix coefficients in a given
row commute pairwise.

The remaining relations \eqref{quater} amount to
\begin{align} \label{rels}& \varepsilon_{ij} \varepsilon_{kl}
(x_{ki}x_{lj} - x_{li}x_{kj}) =
x_{\langle kl \rangle, \langle ij\rangle},  \end{align} where
$1\leq i \neq j \leq 3$, $1\leq k \neq l \leq 3$. Here for $k \not=l$, $\langle
kl\rangle \in \{ 1, 2, 3 \}$ is such that $\{ k, l \} \cup
\{\langle lt\rangle\} = \{ 1, 2, 3 \}$,  and $\varepsilon_{kl}\in \{\pm 1\}$ is defined
by $e_ke_l = \varepsilon_{kl}e_{\langle kl\rangle}$, with $\varepsilon_{kl}=-\varepsilon_{lk}$.

Using the antipode relations \eqref{rels} yield
\begin{align}
\label{rels1.2} &
\varepsilon_{ij} \varepsilon_{kl}
(x_{lj}x_{ki} - x_{li}x_{kj}) =
x_{\langle kl \rangle, \langle ij\rangle}.
\end{align}
Comparing \eqref{rels} and \eqref{rels1.2} we get the  relations between commutators
$$[x_{ki},x_{lj}] = [x_{li},x_{kj}].$$
Applying the antipode this gives
$$[x_{ki},x_{lj}]= [x_{kj},x_{li}],$$
and finally
$$[x_{ki},x_{lj}] = - [x_{lj},x_{kj}]=0,$$
for $k \not = l$ and $i \not =j$.
We conclude that $k[C_A]$ is commutative.




\medbreak Furthermore, the orthogonality of the matrix $(x_{ij})$
and equations \eqref{rels}
for the pairwise commuting elements $x_{ij}$,
give the defining relations for $\Oo[\SO_3(k)] \simeq
\Oo[\PSL_2(k)]$, therefore determining a unique algebra map
\begin{equation}\label{surj}f: \Oo[\PSL_2(k)] \to k[x_{ij}|\, 1\leq i, j \leq 3] =
k[C_A].\end{equation} Because of formula \eqref{delta}, this map
is a bialgebra map, whence a Hopf algebra map. Thus $k[C_A]$ is
isomorphic to $\Oo[\Gamma]$, where $\Gamma \subseteq \PSL_2(k)$ is
an algebraic subgroup, as claimed.
\end{proof}

\subsection{The Hopf algebra $H(n)$} Let $n \geq 2$.
Let  $H(n)$  be the algebra presented by generators $a_{ij}$,
$b_{ij}$, $1\leq i, j \leq n$, and relations
\begin{equation*}
A{}^tB = I = {}^tBA, \quad B{}^tA = I = {}^tAB,
\end{equation*}
where $A$ and $B$ denote, respectively, the matrices $(a_{ij})$
and $(b_{ij})$, and $I$ is the $n\times n$ identity matrix.

The algebra $H(n)$ has a Hopf algebra structure with
multiplication, counit and antipode given, respectively, by
\begin{align*}\Delta(a_{ij}) & =
\sum_la_{il} \otimes a_{lj}, \quad \epsilon(a_{ij}) = \delta_{ij},
\quad \mathcal S(a_{ij}) = b_{ji}, \\
\Delta(b_{ij}) & = \sum_lb_{il} \otimes b_{lj}, \quad
\epsilon(b_{ij}) = \delta_{ij}, \quad \mathcal S(b_{ij}) = a_{ji}.
\end{align*}
The Hopf algebra $H(n)$ coincides with the universal cosovereign
Hopf algebra $H(I_n)$, introduced in \cite{univ-cosov}
(when $k= \C$, it is the Hopf $*$-algebra associated with Wang's universal compact
quantum group in \cite{wa1}). It is
cosemisimple by \cite{bic3}, with $\mathcal S^2 = \id$. The $n$-dimensional vector
space $V$ with basis $(e_i)_{1\leq i\leq n}$ has a natural
$H(n)$-comodule structure $\alpha: V \to V\otimes H(n)$, where
$\alpha(e_i)=\sum_{j=1}^n e_j\otimes a_{ji}$, $1\leq i\leq n$.

Moreover,  for any Hopf algebra $H$ with $\mathcal S^2 = \id$, and
$H$-comodule $\alpha: V \to V \otimes H$, of dimension $\dim V =
n$, there exist an $H(n)$-comodule structure $\alpha': V \to V
\otimes H(n)$, and a unique Hopf algebra map $f: H(n)\to H$ such
that $(\id \otimes f)\alpha' = \alpha$.

\medbreak Consider the Hopf subalgebra $H_+(2)$ of $H(2)$
generated, as an algebra, by all products $a_{ij}b_{tl}$, $1\leq i,j,
t, l \leq 2$. The following result is the algebraic generalization of a result
proved in \cite{ba1} (Theorem 5)
in the compact quantum group framework. This proof given here is simpler.

\begin{corollary} $H_+(2)$  is a commutative Hopf subalgebra
of $H(2)$ isomorphic to $\Oo[\PSL_2(k)]$.
\end{corollary}

\begin{proof} Let $V$ be the standard $2$-dimensional $H(2)$-comodule.
Keep the notation in the proof of Proposition \ref{coefA}. The
subalgebra $H_+(2)$ coincides with the subalgebra $k[C_A]$. We
have  $\mathcal S^2 = \id$ in $H(2)$. Consider the surjective Hopf
algebra map $f: \Oo[\PSL_2(k)] \to H_+(2)$ in \eqref{surj}.

On the other hand, there is a natural Hopf algebra inclusion
$\Oo[\PSL_2(k)] \subseteq \Oo[\SL_2(k)]$. The universal property
of $H(2)$ in \cite[Theorem 4.2]{univ-cosov} gives a surjective
Hopf algebra map $H(2) \to \Oo[\SL_2(k)]$ that induces by
restriction a Hopf algebra map $g: H_+(2) \to \Oo[\PSL_2(k)]$. The
map $g$ defines an inverse of $f$. Thus $f$ is an isomorphism and
we are done.
\end{proof}

\subsection{The finite dimensional case}\label{subgrupos} Let $H$ be a finite dimensional
Hopf algebra. Since $k$ has characteristic zero, then $H$
satisfies $\mathcal S^2 = \id$ if and only if $H$ is cosemisimple,
if and only if $H$ is semisimple \cite{larson-radford}.

We shall use these facts together with Proposition \ref{coefA} to
prove a refinement of \cite[Theorem 11]{NR} in the finite
dimensional context. See Theorem \ref{nr-ref}.

\medbreak Recall that a finite subgroup of $\PSL_2(k)$, called a
\emph{polyhedral group}, is up to isomorphisms, one in following
list.

(i) The cyclic group  $\Z_n$, $n \geq 1$;

(ii) the dihedral group  $D_n$, of order $2n$, $n \geq 2$;

(iii) the group $T$ of symmetries of the regular tetrahedra, of
order $12$;

(iv) the group $O$ of symmetries of the regular octahedra, of
order $24$;

(v) the group $I$ of symmetries of the regular icosahedra, of
order $60$.

\medbreak We have $T \simeq \mathbb A_4$, $O \simeq \mathbb S_4$
and $I \simeq \mathbb A_5$.

\medbreak The Hopf algebra $k^{\Z_n} \simeq k\Z_n$ is
cocommutative. For the other subgroups $\Gamma$, the coalgebra
types of $k^{\Gamma}$ are the following:
$$k^{D_n}: \begin{cases}(1, 2; 2, (n-1)/2), \quad $n$ \text{ odd},\\ (1, 4; 2, (n/2)-1), \quad $n$ \text{ even},\end{cases}$$
$$k^{\mathbb A_4}: (1, 3; 3, 1), \quad k^{\mathbb S_4}: (1, 2; 2, 1; 3, 2), \quad k^{\mathbb A_5}: (1, 1; 3, 2; 4, 1; 5, 1).$$

\begin{proof}[Proof of Theorem \ref{nr-ref}]
Let $V$ be the simple $2$-dimensional comodule
corresponding to $C$. Then $A = V\otimes V^* \simeq M_2(k)$ is an
$H$-comodule algebra. Since $H$ is finite dimensional and
cosemisimple, then $\mathcal S^2 = \id$. By Remark \ref{traza},
the trace map $\Tr: A \to k$ is $H$-colinear.

Applying Proposition \ref{coefA}, we get that $k[C_A] \subseteq H$
is a commutative Hopf subalgebra of $H$ isomorphic to
$k^{\Gamma}$, where $\Gamma$ is a finite subgroup of $\PSL_2(k)$.
Since $A = V\otimes V^*$ as $H$-comodules, the subcoalgebra $C_A$
coincides with $C \mathcal S(C)$.

It remains to show that $\Gamma$ is a non cyclic subgroup, which
will ensure, in view of the above list, that it has even order. So
assume that $\Gamma = \langle g \rangle$, where $g \in \PSL_2(k)$
has finite order. Let $M \in \SL_2(k)$ be a matrix representing
$g$: it follows from the fact that $g$ has finite order that $M$
is diagonalisable, and hence the conjugation action of $M$ on $A =
V\otimes V^*$ has at least two trivial subrepresentations, which
contradicts the fact that $V$ is a simple $H$-comodule (since the
category of $k[C_A]$-comodules is a full subcategory of the
category of $H$-comodules).

Consider now the subgroup $G[\chi]$, where $\chi \in C$ is the
character of $V$.  Then we see, by examining the decomposition of $V \otimes V^*$ into a direct sum of irreducibles,
that $|G[\chi]| = 1, 2$ or $4$. More precisely we have

\begin{enumerate}
\item $|G[\chi]| = 1 \iff V\otimes V^* \simeq  k\oplus W$, where $W$ is a simple $k^\Gamma$-comodule of dimension $3$.
\item $|G[\chi]| = 2 \iff V\otimes V^* \simeq  k\oplus kx \oplus Z$, where $x \in G(H)$
has order $2$ (in fact $x \in G[\chi]$) and $Z$ is a simple $k^\Gamma$-comodule of dimension $2$.
\item $|G[\chi]| = 4 \iff V\otimes V^* \simeq  k\oplus kx \oplus ky \oplus kz$ where
$x,y, z \in G(H)$ (and then $G[\chi]= \{1,x,y,z\}$).
\end{enumerate}
The reader will easily check the details (see \cite{NR}, Theorem
10, for this kind of reasoning). The proof of the last statement
of Theorem 1.1 now follows from the coalgebra types of $k^\Gamma$,
listed above. \end{proof}





\begin{remark} Keep the notation in Theorem \ref{nr-ref}.
Assume that $\chi\chi^* = 1 + \lambda$, where $\lambda$ is an
irreducible character of degree $3$. This means that possibility
(1) in \cite[Theorem 11]{NR} does not hold.

Then $B = k[C\mathcal S(C)]$ coincides with the subalgebra
generated by the subcoalgebra $C_{\lambda}$ containing $\lambda$,
of dimension  $12$, $24$ or $60$. Further, assume that $\dim
k[C_{\lambda}] = 24$, that is, $k[C_{\lambda}] \simeq k^{\mathbb
S_4}$. In view of the fusion rules of representations of $\mathbb
S_4$, $\chi \notin k[C_{\lambda}]$. The coalgebra types in the
remaining cases, imply that $\chi \notin k[C_{\lambda}]$ for any
possible isomorphism class of $k[C_{\lambda}]$. In particular,
$k[C_{\lambda}]$ is a proper Hopf subalgebra of $H$.
\end{remark}

\subsection{A special situation}
We use the previous considerations to attach an exact
sequence to a cosemisimple Hopf algebra having a faithful comodule of dimension $2$, in a special
situation. The precise result is the following one.

Recall from \cite[4.1]{ksz} that the order of an $H$-comodule $V$
is the smallest natural number $n$ such that $V^{\otimes n}$
contains the trivial $H$-comodule. For instance, if $V$ is
self-dual and nontrivial, then the order of $V$ equals $2$.

\begin{theorem}\label{shortspecial}
Let $H$ be a cosemisimple finite dimensional Hopf
algebra having a faithful irreducible comodule V of dimension $2$, and assume that $V \otimes V^*\simeq V^*
\otimes V$.
Then there is a cocentral abelian exact
sequence of Hopf algebras
$$k\to k^\Gamma \to H \to  k\Z_m\to k$$
for  a polyedral group $\Gamma$ of even order and $m\geq 1$, such that $m$ divides the order of $V$.
\end{theorem}

\begin{proof}
The universal property of $H(2)$ yields a surjective Hopf algebra
map $H(2) \to H$. Denote by $H_+$ the image of $H_+(2)$ in $H$.
We know from Theorem \ref{nr-ref} and its proof that $H_+ \simeq k^\Gamma$
  for  a polyedral group $\Gamma$ of even order.
If $C$ is the coalgebra corresponding to $V$, we have
$H_{\coad} = k[C\mathcal S(C)]=H_+$ by Lemma \ref{chichi*}, and hence
the exact sequence \eqref{univcocentral} gives a cocentral exact sequence
$$k\to k^\Gamma \to H \to  kM\to k$$
for a (finite) group $M$. It remains to check that the group $M$ is cyclic.
We have $kM \simeq H/(H_+)^+H$, and hence it is enough
to check that $H/(H_+)^+H$ is the group algebra of a cyclic group.
For notational simplicity, the elements $a_{ij},b_{ij}$ of $H(2)$ are still denoted
by the same symbol in $H/(H_+)^+H$. In $H/(H_+)^+H$, we have
$$a_{ii}b_{jj}=1, \quad a_{ij}b_{kl}=0 \ {\rm if} \ i\not= j \ {\rm or} \ k \not= l.$$
Hence if $i \not = j$, we have $a_{ij}=a_{ij}a_{ii}b_{jj}=0$, and
the elements $a_{ii}$ are group-like, and they generate $M$. We
have moreover $a_{ii} = a_{ii}b_{jj}a_{jj}=a_{jj}$, and hence $M$
is cyclic.

\medbreak Let now $N$ be the order of $V$. Let $v_1, v_2$ be a
basis of $V$ such that $\rho(v_i) = \sum_jv_j\otimes a_{ji}$.
Since $V^{\otimes N}$ contains the trivial corepresentation, there
exists $v = \sum_{i_1, \dots, i_N}\lambda_{i_1, \dots, i_N}
v_{i_1} \otimes \dots \otimes v_{i_N}$  in $V^{\otimes N}$ such
that $v\neq 0$ and $$\rho(v) = \sum_{i_1, \dots, i_N, j_1, \dots,
j_N}\lambda_{i_1, \dots, i_N} \, v_{j_1} \otimes \dots \otimes
v_{j_N} \otimes a_{j_1i_1} \dots a_{j_Ni_N} = v \otimes 1.$$ Then,
in $V^{\otimes N} \otimes H/(H_+)^+H$, we have
$$v \otimes 1 = \sum_{i_1, \dots, i_N}\lambda_{i_1, \dots, i_N} \, v_{i_1} \otimes \dots \otimes
v_{i_N} \otimes a^N = v \otimes a^N,$$ where $a = a_{i_1i_1} =
\dots = a_{i_Ni_N}$. Hence $a^N = 1$, and since $m$ equals the
order of $a$, then $m$ divides $N$, as claimed.
\end{proof}

\begin{remark}
The assumption of the previous theorem holds if $H$ is coquasitriangular, or
if $V\simeq V^*$. The self-dual situation is studied in the next sections;
we have in this case $m\leq 2$.
\end{remark}

\begin{corollary}Let $H$ and $V$ be as in Theorem \ref{shortspecial}. Then for all
irreducible $H$-module $W$, the dimension of $W$ divides the order of $V$.\end{corollary}

\begin{proof}As an algebra, $H$ is a crossed product $k^{\Gamma}\#k\Z_m$. Then the irreducible representations
of $H$ are of the form $W_s : = \Ind_{k^{\Gamma}\otimes k C_s}^H
s\otimes U_s$, where $s$ runs over a set of representatives of the
orbits of the action of $\Z_m$ in $\Gamma$, and $C_s \subseteq \Z_m$ is the stabilizer
of $s$, and $U_s$ is an irreducible representation of $C_s$. See
\cite{MoW} for a description of the irreducible representations of
a general group crossed product. In particular, we have $\dim W_s = [\Z_m: C_s] \dim U_s$, and since $\dim U_s$ divides the order of $C_s$, then  $\dim W_s$ divides $m$. This implies the corollary since, by Theorem \ref{shortspecial}, $m$ divides the order of $V$.
  \end{proof}

\begin{remark}\label{plusgene}
The exact sequence \eqref{univcocentral} is still valid
if the cosemisimple Hopf algebra $H$ is not assumed to be finite dimensional,
and hence
we still have a cocentral exact sequence
$$k\to \mathcal O[G] \to H \to  kM\to k$$
where $G \subset \PSL_2(k)$ is an algebraic subgroup and $M$ is a cyclic group.
Note however that in the infinite dimensional situation, the
Hopf algebra $H$ will not necessarily be a crossed product, because
  $H$ is not necessarily free as a left
(or right) module over its Hopf subalgebra $\mathcal O[G]$ \cite{radford,
waterhouse}.
\end{remark}

\section{Hopf algebras with a self-dual faithful comodule}\label{two}

In this section we discuss general results for
Hopf algebras having a faithful self-dual comodule.

\subsection{The Hopf algebra of a nondegenerate bilinear form} Let
$n\geq 2$ and let $E\in \GL_n(k)$. Let  $\B(E)$ be the algebra
presented by generators $a_{ij}$, $1\leq i, j \leq n$, and
relations
\begin{equation*}
E^{-1}\,^tAEA=I=AE^{-1}\,^tAE,
\end{equation*}
where $A$ denotes the matrix $(a_{ij})$ and $I$ is the $n\times n$
identity matrix.

There is a Hopf algebra structure on $\B(E)$ with multiplication,
counit and antipode given, respectively, by
$$\Delta(a_{ij}) = \sum_la_{il} \otimes
a_{lj}, \quad \epsilon(a_{ij}) = \delta_{ij}, \quad \mathcal S(A)
= E^{-1}\,^tAE.$$ The Hopf algebra $\B(E)$ was introduced in
\cite{DL} as the \textit{quantum group of a nondegenerate bilinear
form}. The $n$-dimensional vector space $V = V^E$ with basis
$(e_i)_{1\leq i\leq n}$ has a natural $\B(E)$-comodule structure
$\alpha: V \to V\otimes \B(E)$, where $\alpha(e_i)=\sum_{j=1}^n
e_j\otimes a_{ji}$, $1\leq i\leq n$. Moreover, the nondegenerate
bilinear form $\beta: V\otimes V \to k$, $\beta(e_i\otimes e_j) =
E_{ij}$, is a $\B(E)$-comodule map.

Moreover, $\B(E)$ is universal with respect to this property.
Namely, for any Hopf algebra $H$ and $H$-comodule $V$ of finite
dimension such that $\beta:V\otimes V\to k$ is a nondegenerate
$H$-invariant bilinear form, then $V$ is a $\B(E)$-comodule, where
$E\in \GL_n(k)$ is any matrix associated to $\beta$. In this case,
$\beta$ is a $\B(E)$-comodule map and there exists a unique Hopf
algebra map $f: \B(E)\to H$ such that $(\id \otimes
f)\alpha=\alpha'$, where $\alpha$ and $\alpha'$ are the coactions
on  $V$ of $\B(E)$ and $H$, respectively.

For an $H$-comodule $V$, the existence of a nondegenerate
$H$-invariant bilinear form on $V$ means exactly that $V$ is
self-dual, that is, $V^* \simeq V$ as $H$-comodules. This is also
equivalent to $\mathcal S(C) = C$, where $C \subseteq H$ is the
subcoalgebra spanned by the  matrix coefficients of $V$.

\medbreak Isomorphism classes of the Hopf algebras $\mathcal B(E)$
are classified by the orbits of the action of $\GL_n(k)$ on itself
defined by $G.E = {}^tGEG$, $G, E \in \GL_n(k)$. More precisely,
it is shown in \cite[Theorem 5.3]{B} that for $E\in \GL_n(k)$,
$F\in \GL_m(k)$, the Hopf algebras $\mathcal B(E)$ and $\mathcal
B(F)$ are isomorphic if and only if $n = m$ and $E$ and $F$ are
conjugated under the action of $\GL_n(k)$.

Our motivation for studying the Hopf algebras $\B(E)$ is the following basic result.

\begin{proposition}\label{moti}
Let $H$ be a Hopf algebra having a faithful self-dual $H$-comodule $V$.
\begin{enumerate}
\item There exists a surjective Hopf algebra map $\B(E) \to H$, where the size of the matrix
$E$ is the dimension of $V$.
\item If $H$ is cosemisimple and $\mathcal S^2={\rm id}$, then the above matrix $E$ is either symmetric or skew-symmetric.
\item\label{semisimple} If $H$ is semisimple, then the above matrix $E$ is either symmetric or skew-symmetric.
\end{enumerate}
\end{proposition}

\begin{proof}
The first assertion follows from the previous discussion and the universal
property of $\B(E)$. The second one follows from the Frobenius-Schur theorem
for cosemisimple Hopf algebras with $\mathcal S^2={\rm id}$ \cite{binyjm}, and the
third one follows from the second one and the fact that a semisimple Hopf algebra
is automatically cosemisimple and with $\mathcal S^2={\rm id}$ \cite{larson-radford}
(one may also use directly the Frobenius-Schur theorem for semisimple Hopf algebras
from \cite{li-mo}).
\end{proof}

\begin{remark} Consider the situation in Proposition
\ref{moti} (\ref{semisimple}), where $V$ is an irreducible
comodule. Then the matrix $E$ is symmetric or skew-symmetric, if
and only if the Frobenius-Schur indicator $\nu(V)$ equals $1$ or
$-1$, respectively. \end{remark}

\subsection{Exact sequences attached to quotients of $\B(E)$}\label{plusmin}

Let $L \subseteq \B(E)$ denote the linear span of
$a_{ij}$, $1 \leq i, j \leq n$. Let also $\B_+(E) = k[L^2]$ be the
subalgebra generated by $L^2$, and $\B_-(E)$ be the subspace of
odd powers of $L$.

Then $\B_+(E)$ is a Hopf subalgebra, $\B_-(E)$ is a subcoalgebra,
and there is a $\Z_2$-algebra grading $\B(E) = \B_+(E) \oplus
\B_-(E)$ (the fact that the sum is direct follows
from the $\Z_2$-action constructed in proof of the next result).

\begin{proposition}\label{exacta-b+} The Hopf algebra $\B_+(E)$ is a normal Hopf subalgebra of
$\B(E)$ and there is a cocentral exact sequence of Hopf algebras
$$k \to \B_+(E) \to \B(E) \to k\Z_2 \to k.$$
\end{proposition}

\begin{proof}
It is easy to check the existence of a Hopf algebra map $\zeta: \B(E) \to k\Z_2$ be defined by $\zeta(a_{ij}) =
\delta_{ij}g$, where $1\neq g \in \Z_2$.
The relation $\zeta(x_1) \otimes x_2 = \zeta(x_2) \otimes
x_1$ is clearly satisfied for $x = a_{ij}$, $1\leq i, j \leq n$.
Since these generate $\B(E)$, we get that the map $\zeta$ is
cocentral.

The  Hopf algebra map $\zeta$ induces a $\Z_2$-action on $\B(E)$, given by
$g.a_{ij} = -a_{ij}$. It is then clear that $\B_+(E) = \B(E)^{\Z_2}= \B(E)^{\rm co \zeta}$, and
we have our exact sequence by Lemma \ref{exactness}.
\end{proof}

Now let $\B(E) \to H$ be a Hopf algebra quotient. We denote by $H_+$ the image
of $\B_+(E)$ in $H$, this is the subalgebra of $H$ generated by the coefficients
of $V^{\otimes 2}$.

\begin{proposition}
\label{exacta-gene}
Let $\B(E) \to H$ be a Hopf algebra quotient. Then $H_+$ is a normal Hopf subalgebra
of $H$, and we have $[H:H_+] = \dim H/(H_+)^+H = 1$ or $2$. Moreover if $[H:H_+]=2$, we have
a
cocentral short exact sequence of Hopf algebras $$k \to H_+ \to H
\to k\Z_2 \to k$$
\end{proposition}

\begin{proof}
It is clear that $H_+$ is a normal Hopf subalgebra of $H$.
For notational simplicity, the elements $a_{ij}$ of $\B(E)$ are still denoted
by the same symbol in any quotient. In $H/(H_+)^+H$, we have
$$a_{ii}a_{jj}=1, \quad a_{ij}a_{kl}=0, \ {\rm if} \ i\not= j \ {\rm or} \ k \not= l$$
Hence if $i \not = j$, we have $a_{ij}=a^2_{ii}a_{ij}=0$, and $a_{ii}=a_{ii}a_{jj}^2=a_{jj}$.
It follows that $H/(H_+)^+H \simeq k$ or $H/(H_+)^+H \simeq k\Z_2$, depending
whether $g=a_{ii}=1$ or not.
If $H/(H_+)^+H \simeq k\Z_2$, the canonical projection $H \to H/(H_+)^+H$
is defined by  the same formula as the map $\zeta$ in the proof of the previous proposition,
and one concludes in a similar manner.
\end{proof}

\begin{remark}
We gave the direct (easy) proofs for the  exact sequences in the previous propositions.
These exact sequences are special cases of (\ref{univcocentral}).
\end{remark}

\subsection{$\Z_2$-grading on comodule categories and adjoint subcategory}\label{z2grading}


By
\cite[Theorem 1.1]{B} there is an equivalence of tensor categories
\begin{equation}\label{equiv}\B(E)-\comod \simeq \Oo_q[\SL_2(k)]-\comod,\end{equation} where $q$ is a
root of the polynomial $X^2 + \Tr(E{}^tE^{-1})X + 1$. The
equivalence sends the standard $\B(E)$-comodule $V^E$ to the
standard $2$-dimensional $\Oo_q[\SL_2(k)]$-comodule $U$.

\medbreak We shall assume in this subsection that the roots of the
polynomial $X^2 + \Tr(E{}^tE^{-1})X + 1$ are generic. In this case
recall that $\Oo_q[\SL_2(k)]$-comod is a semisimple tensor
category whose simple objects are represented by the comodules
$(U_n)_{n \geq 0}$, where $\dim U_n = n + 1$ (in particular, $U_0
= k$ and $U_1 = U$ is the standard $2$-dimensional comodule).
These obey the following fusion rules:
\begin{equation}\label{fusionrules}U_i \otimes U_j \simeq U_{i+j}
\oplus U_{i+j-2} \oplus \dots  \oplus U_{|i-j|},\end{equation} for
all $i, j \geq 0$.

Let $\mathcal C = \Oo_q[\SL_2(k)]-\comod$. Consider the full
abelian subcategories $\mathcal C_+$, respectively $\mathcal C_-$,
whose simple objects are $(U_i)_{i\text{ even}}$, respectively
$(U_i)_{i\text{ odd}}$. It is well-known that the fusion rules \eqref{fusionrules} imply
that $\mathcal C = \mathcal C_+ \oplus \mathcal C_-$ is a faithful
$\Z_2$-grading on $\mathcal C$. In particular, $\mathcal C_+$ is a
full tensor subcategory of $\mathcal C$.


Furthermore, $\mathcal C_+ = \mathcal C_{\ad}$ is the tensor subcategory of $\mathcal C$ generated by $U_2$. 
In particular, since $U_2 \otimes U_2 \simeq U_4 \oplus U_2 \oplus
U_0$, we have $(\mathcal C_{\ad})_{\ad} = \mathcal C_{\ad}$.

\begin{remark} Let  $\Cc$ be a $k$-linear  abelian semisimple tensor category. 
Suppose $\mathcal C$ is generated by a  self-dual
simple object $V$. Then there is a (not necessarily faithful) $\Z_2$-grading $\mathcal C =
\mathcal C_+ \oplus \mathcal C_-$, where $\mathcal C_+ = \mathcal
C_{\ad}$. 

Indeed, $\mathcal C_{\ad}$ is the full tensor subcategory generated by the simple constituents of $V^{\otimes 2}$.
If $V$ is in $\mathcal C_{\ad}$, then $\mathcal C = \mathcal C_{\ad}$. Otherwise, the odd tensor powers of $V$ generate a non-trivial component $\mathcal C_-$ of the grading. \end{remark}

\medbreak In what follows we discuss some relations between the adjoint $\mathbb Z_2$-grading in categories of comodules and the Hopf subalgebras $H_{\pm}$ defined in Subsection \ref{plusmin}.

Let $\mathcal C^E$ denote the category of finite dimensional
$\B(E)$-comodules. The equivalence \eqref{equiv} has the following
corollary.

\begin{corollary}\label{z2-grad} There is a faithful $\Z_2$-grading $\mathcal C^E = \mathcal C^E_+ \oplus \mathcal
C^E_-$, where $\mathcal C^E_+ = \mathcal C^E_{\ad}$ is the full
tensor subcategory generated by the simple comodule $V^E_2$. Moreover
$\B_{\pm}(E)-\comod = \mathcal
C^E_{\pm}$.
\end{corollary}

\begin{proof}
The first assertion follows from the equivalence \eqref{equiv}.
The adjoint category $\mathcal C^E_+ = \mathcal C^E_{\ad}$ is the
category of corepresentations of a Hopf subalgebra $B$ of $\B(E)$.
Indeed, it follows from the first assertion that $B$ is the
subalgebra generated by the matrix coefficients of the simple
comodule $V^E_2$. In view of the relation $V^E \otimes V^E \simeq
k \oplus V^E_2$, $B$ coincides with the subalgebra of $\B(E)$
generated by the products $a_{ij}a_{kl}$,  $1 \leq i, j, k, l \leq
n$; that is, $B = \B_+(E)$. Similarly, $\B_-(E)-\comod = \mathcal
C^E_-$. \end{proof}

In the next proposition we shall consider the one-parameter
deformation of $\SO_3(k)$ as studied by Takeuchi in
\cite{takeuchi-so3}. See also \cite{frt, podles-spheres}.

We use here the notation $\Oo_{q^2}[\SO_3(k)]$ to indicate the
Hopf algebra denoted $\A_{q^2}(3)$ in \cite{takeuchi-so3}.

\begin{proposition}\label{b+} There is an isomorphism of Hopf algebras $\Oo_{q}[\SL_2(k)]_+ \simeq
\Oo_{q^2}[\SO_3(k)]$. Therefore the equivalence \eqref{equiv}
restricts to an equivalence of tensor categories
\begin{equation}\label{equiv+}\B_+(E)-\comod \simeq
\Oo_{q^2}[\SO_3(k)]-\comod.\end{equation} \end{proposition}

Recall that in \eqref{equiv} $q$ is a root of the polynomial $X^2
+ \Tr(E{}^tE^{-1})X + 1$, and we are assuming that $q$ is generic.

\begin{proof} Note that $\Oo_{q}[\SL_2(k)] = \B(E_q)$, where $E_q$
is the $2\times 2$-matrix defined by $(E_q)_{11} = (E_q)_{22} =
0$, $(E_q)_{12} = -q^{-1/2}$, $(E_q)_{21} = q^{1/2}$. So that,
denoting by $x_{ij}$, $1\leq i, j \leq 2$, the generators of
$\Oo_{q}[\SL_2(k)]$, we have that $\Oo_{q}[\SL_2(k)]_+ =
\B_+(E_q)$ is the Hopf subalgebra generated by the products
$x_{ij}x_{tl}$, $1 \leq i, j, t, l \leq 2$.

Since the parameter $q$ is generic by assumption, there is an
injective Hopf algebra map $f: \Oo_{q^2}[\SO_3(k)] \to
\Oo_{q}[\SL_2(k)]$ \cite[Proposition 8]{takeuchi-so3},
\cite{dijkhuizen}. Furthermore, in view of the definition of $f$
in \cite{takeuchi-so3}, the image of $f$ coincides with
$\Oo_{q}[\SL_2(k)]_+ = \B_+(E_q)$, therefore defining an
isomorphism $\Oo_{q}[\SL_2(k)]_+ \simeq \Oo_{q^2}[\SO_3(k)]$.

Finally, note that the equivalence \eqref{equiv}, being an
equivalence of tensor categories, must preserve the adjoint
subcategories. Then it induces by restriction the equivalence
\eqref{equiv+}. This finishes the proof of the proposition.
\end{proof}

\begin{proposition} \label{genegrading}Let $\mathcal C$ be the tensor category of finite dimensional $H$-comodules, where
$H$ is a semisimple Hopf algebra having a self-dual faithful
simple comodule $V$.  Then  $\mathcal C_+=H_+-\comod$. \end{proposition}

\begin{proof} We have a surjective Hopf algebra map $f: \B(E) \to H$. This induces
a dominant tensor functor $F: \mathcal C^E \to \mathcal C$.
Therefore, in view of Corollary \ref{z2-grad}, $\mathcal C =
F(\mathcal C^E_{+}) \oplus F(\mathcal C^E_{-})$ is a $\Z_2$-grading on $\mathcal
C$. Since $F(V^E) = V$, then it is clear that $F(\mathcal C^E_{+})$ is generated by
the simple constituents of $V^{\otimes 2}$, whence
$F(\mathcal C^E_{+}) = \mathcal C_{\ad}=H_+-\comod$. \end{proof}

\begin{remark} Consider the dominant tensor functor $F: \mathcal C^E \to \mathcal C$ in the
proof of Proposition \ref{genegrading}. In view of the
equivalences \eqref{equiv} and \eqref{equiv+}, $F$ gives rise to
dominant tensor functors
\begin{equation*}\Oo_{q}[\SL_2(k)]-\comod \to H-\comod, \quad  \Oo_{q^2}[\SO_3(k)]-\comod \to
H_+-\comod, \end{equation*} for an appropriate choice of the
(generic) parameter $q \in k^{\times}$. \end{remark}

\section{The $2\times 2$ case}\label{2por2}

\subsection{Classification of $\B(E)$'s} In the $2\times 2$ case, the orbits of the
$\GL_2(k)$-action are
represented by one the matrices \begin{align*} E_{\mathfrak m} = \left(\begin{matrix}0 & -1\\
1 & 1 \end{matrix}\right), \; E_1 = \left(\begin{matrix}0 & -1\\
1 & 0 \end{matrix}\right), \; E_q = \left(\begin{matrix}0 & -q^{-\frac{1}{2}}\\
q^{\frac{1}{2}} & 0 \end{matrix}\right), \, q\in k\backslash\{0,
1\}.\end{align*} Moreover, the matrices $E_q$  and $E_{q'}$ belong
to the same orbit if and only if $q'  = q^{\pm 1}$.

\medbreak Hence in this case the Hopf algebra $\B(E)$ is
isomorphic to exactly one of the Hopf algebras in the following
list:
\begin{enumerate}
\item[(1)] $\Oo_q[\SL_2(k)]$, $q\in
k\backslash\{0, 1\}$, corresponding to the matrices $E_q$;
\item[(2)] $\Oo[\SL_2(k)]$, corresponding to the
matrix $E_1$; \item[(3)] $\B(E_{\mathfrak m})$.
\end{enumerate}
See  \cite[Theorem 5.3]{B}, \cite[Section 6]{DL}. Note that, by
\cite[Theorem 1.1]{B}, the category of comodules over the Hopf
algebra $\B(E_{\mathfrak m})$ is tensor equivalent to the category
of $\Oo[\SL_2(k)]$-comodules.

\medbreak In particular, we have the following

\begin{corollary}\label{sym-ssym} Suppose that the matrix $E$ is symmetric (respectively,
skew-symmetric). Then the Hopf algebra $\B(E)$ is isomorphic to
$\Oo_{-1}[\SL_2(k)]$ (respectively, to $\Oo[\SL_2(k)]$).\qed
\end{corollary}


\begin{example}\label{stefan}
Let $H$ be a finite dimensional Hopf algebra and let $C\subseteq
H$ be a simple subcoalgebra such that  $\dim C = 4$ and $\mathcal
S(C) = C$. Suppose that the Hopf subalgebra $k[C]$ is not
semisimple. Then the order of $S^2\vert_C$ is an integer $m > 1$
\cite{larson-radford}. The assumptions imply that there is a Hopf
algebra map $\B(E)\to H$, whose image is $k[C]$, for some $E \in
\GL_2(k)$. Since $k[C]$ is not cosemisimple, then  $\B(E)\simeq
\Oo_q(\SL_2(k))$, for some root of unity $q\in k^{\times}$, $q\neq
\pm 1$.

This fact was proved by Stefan in \cite[Theorem 1.5]{stefan}, and
then it was used in the classification of Hopf algebras of low
dimension \cite{stefan, dim12}. Note that when $q$ is a root of
unity of odd order, $\Oo_q(\SL_2(k))$ contains a central Hopf
subalgebra $A\simeq \Oo(\SL_2(k))$. This gives a central Hopf
subalgebra $\overline A \subseteq k[C]$, such that the quotient
$H/HA^+$ is dual to a pointed Hopf algebra.
\end{example}

The following corollary gives the classification of cosemisimple
Hopf algebras satisfying condition \eqref{4} in the
Nichols-Richmond Theorem \cite{NR} under the assumption that the
simple $2$-dimensional comodule is self-dual and faithful.

\begin{corollary}\label{nr-inf} Let $H$ be a cosemisimple
Hopf algebra. Suppose $H$ has a family $\{C_n: \, n\geq 1\}$ of
simple subcoalgebras such that $\dim C_n = n^2$,  satisfying, for
all $n\geq 2$,
\begin{equation*}C_n\mathcal S(C_2)  = C_{n-1} + C_{n+1}, \; n \text{
even,}\qquad  C_nC_2  = C_{n-1} + C_{n+1}, \; n \text{
odd.}\end{equation*}

Assume in addition that $H = k[C]$, where $C
= C_1$, and $\mathcal S(C) = C$. Then $H$ is isomorphic to one of
the Hopf algebras $\Oo_{\pm1}[\SL_2(k)]$, $\B(E_{\mathfrak m})$, or
$\Oo_q[\SL_2(k)]$, $q\in k\backslash\{0, \pm1\}$ not a root of unity.
\end{corollary}

\begin{proof} Let $\mathcal G(H^*)$ denote the Grothendieck
ring of the category of finite dimensional $H$-comodules. The
assumption $H = k[C]$ implies that  $\mathcal G(H^*)$ is spanned
by the irreducible characters $\chi_i$ corresponding to the simple
subcoalgebras $C_i$, $i \geq 1$. The assumption also implies that
there is a surjective Hopf algebra map $f: \B(E)\to H$, for some
$E \in \GL_2(k)$, such that $f(\overline C) = C$, where $\overline
C \subseteq \B(E)$ is the simple subcoalgebra corresponding to the
standard $2$-dimensional corepresentation $V$ of $\B(E)$.
Furthermore, $f$ induces an isomorphism between the Grothendieck
rings. Then $f$ is an isomorphism, by \cite[Theorem 1.2]{B}. The
result follows from the classification of the Hopf algebras $\B(E)$
in the $2\times 2$ case. \end{proof}

\subsection{The Hopf algebra $\Oo_{-1}[\SL_2(k)]$}\label{three}

Recall from \cite{ks} the definition of the quantum groups
$\Oo_q[\SL_2(k)]$, where $q\in k$, $q\neq 0$. We shall be
interested in the case $q = -1$.

As an algebra, $\Oo_{-1}[\SL_2(k)]$ is presented by generators
$a$, $b$, $c$ and $d$, with relations
$$\quad bc = cb, \quad ad =
da,$$ \begin{equation}\label{relaciones}ba = - ab, \quad ca = -
ac, \quad db = - bd,\quad dc = - cd, \end{equation}$$ad + bc =
1.$$ We shall use the notation $\A = \Oo_{-1}[\SL_2(k)]$.
Similarly, $A$ will denote the commutative Hopf algebra
$\Oo[\SL_2(k)]$, with standard commuting generators $x, y, z, t$
and defining relation $xz-yt = 1$.

The $i,j$ coefficient of the matrix $\left(\begin{matrix}a & b \\
c & d \end{matrix}\right)$ (respectively, $\left(\begin{matrix}x &
y \\ z & t \end{matrix}\right)$) will be alternatively denoted by
$a_{ij}$ (respectively, $x_{ij}$), $1\leq i, j\leq 2$.

The algebra $\A$ has a Hopf algebra structure with
comultiplication, counit and antipode determined by the formulas
\begin{equation}\label{delta-antipoda}\Delta(a_{ij})
= \sum_ka_{ik} \otimes a_{kj}, \quad \epsilon(a_{ij}) =
\delta_{ij},\quad \mathcal S\left(\begin{matrix}a & b \\ c & d
\end{matrix}\right) = \left(\begin{matrix}d & b \\ c & a \end{matrix}\right).\end{equation}

The following lemma is well-known. See \cite[Proposition
2.2]{podles}.

\begin{lemma}\label{irreps} Let $V$ be an irreducible representation of
$\A$. Then the dimension of $V$ equals $1$ or $2$. \qed
\end{lemma}

Consider the $\Z_2$-algebra grading $\A = \A_+ \oplus \A_-$ as in
Subsection \ref{z2grading}. The Hopf subalgebra $\A_+$ is the
largest commutative Hopf subalgebra of $\A$ \cite[Theorem
15.4]{takeuchi2}.

In the same way, we have a $\Z_2$-algebra grading $A = A_+ \oplus
A_-$. By \cite[Corollary 15.3]{takeuchi2},  $A_+ \simeq \A_+$ as
Hopf algebras.

The category $\comod-\A_+$ corresponds to the tensor subcategory
of $A$-comodules having only even weights \cite[Theorem
15.5]{takeuchi2}. As Hopf algebras, $A_+$ is isomorphic to the
algebra of coordinate functions on the group $\PSL_2(k) \simeq
\SO_3(k)$. Note, however, that $A$ is \emph{not} free as a left
(or right) module over its Hopf subalgebra $A_+$ \cite{radford,
waterhouse}.

\medbreak There is a $\Z_2 \times \Z_2$-algebra grading on $\A$
(respectively, on $A$), such that $a, d$ (respectively, $x, t$)
have degree $(1, 0)$, and $b, c$ (respectively, $y, z$) have
degree $(0, 1)$. Associated to this grading, there is a deformed
product $\centerdot$ in $A$, defined in the form
$$u\centerdot v = (-1)^{j(k+l)} uv, $$ for homogeneous elements $u, v \in
A$, of degrees $(i, j)$ and $(k, l)$, respectively. See
\cite[Section 15]{takeuchi2}.

This defines an associative product which makes $(A, \centerdot)$
into a $\Z_2\times \Z_2$-graded algebra with the same unit as $A$.
Moreover, there is an isomorphism of graded algebras $\lambda: (A,
\centerdot)  \to \A$, determined by $\lambda(x_{ij}) = a_{ij}$,
$1\leq i, j\leq 2$. The vector space isomorphism underlying
$\lambda$ is furthermore a coalgebra isomorphism \cite[Proposition
15.1]{takeuchi2}.

\smallbreak We have $\lambda(A_{\pm}) = \A_{\pm}$. Moreover, the
product in $A_+ \subseteq A$ coincides with the deformed product
$\centerdot$, since $A_+$ sits in degrees $(0, 0)$ and $(1, 1)$.
Hence the restriction of $\lambda$ induces a canonical Hopf
algebra isomorphism $A_+ \to \A_+$.

\medbreak Recall that, by Proposition \ref{exacta-b+}, $\A_+$ is a
normal Hopf subalgebra of $\A$, and we have a short exact sequence
of Hopf algebras $k \to \A_+ \to \A \overset{\zeta}\to k\Z_2 \to
k$, where the map  $\zeta: \A \to k\Z_2$ is determined by
$\zeta(a_{ij}) = \delta_{ij}g$, $1\neq g \in \Z_2$.

\begin{proposition}\label{extension} $\lambda$ induces a commutative diagram
with exact rows \begin{equation}\label{comm-diag}\begin{CD}k @>>> A_+ @>>> A @>{\zeta'}>> k\Z_2 @>>> k \\
@VVV @VV{\simeq}V        @VV{\lambda}V @VV{=}V @VVV \\
k @>>> \A_+ @>>> \A @>{\zeta}>> k\Z_2 @>>> k.
\end{CD}\end{equation} \end{proposition}

\begin{proof} The proof is straighforward. \end{proof}

\begin{remark}\label{correspondencia-lambda}
The $\Z_2$-gradings $A = A_+ \oplus A_-$ and $\A = \A_+ \oplus
\A_-$ are associated to the $\Z_2$-coaction arising from the Hopf
algebra surjection $\zeta': A \to k\Z_2$, respectively, $\zeta: \A
\to k\Z_2$. Moreover, $\lambda: A \to \A$ is a $\Z_2$-graded map.

On the other hand, we have $A_+ = A_{(0, 0)} \oplus A_{(1, 1)}$,
$A_- = A_{(1, 0)} \oplus A_{(0, 1)}$ and similarly, $\A_+ =
\A_{(0, 0)} \oplus \A_{(1, 1)}$, $\A_- = \A_{(1, 0)} \oplus
\A_{(0, 1)}$. This can be seen using that the set
$$\{a^ib^jc^k: 0 \leq  i, j, k\} \cup \{b^ic^jd^k: 0 \leq i, j, \,
0 < k\},$$ forms a basis of the algebra $\A$, and similarly for
$A$. See \cite{ks}.

Suppose $\J \subseteq \A$ is a $\Z_2$-homogeneous ideal of $\A$,
that is, $\J$ is an ideal and $\J = \J_+ \oplus \J_-$, where
$\J_{\pm} = \J \cap \A_{\pm}$. Then it follows from the definition
of the product $\centerdot$ that $(\J, \centerdot)$ is an ideal of
$(\A, \centerdot)$. Therefore, $J = \lambda^{-1}(\J)$ is a
$\Z_2$-homogeneous ideal of $A$. \end{remark}

\subsection{Finite quantum subgroups of
$\Oo_{-1}[\SL_2(k)]$}\label{s-5}

Let $H$ be a nontrivial (i.e. non commutative) finite dimensional quotient Hopf algebra
of $\A = \Oo_{-1}[\SL_2(k)]$. In particular, $H$ is semisimple and
cosemisimple.


\begin{proposition}\label{extension2}
The Hopf algebra $H$
fits into an abelian cocentral exact sequence
\begin{equation}\label{exacta2}k \to k^{\Gamma} \to H \to k\Z_2 \to
k, \end{equation} where $\Gamma$ is a finite subgroup of even order
of
$\PSL_2(k)$. The adjoint Hopf subalgebra $H_{\coad}$ is
commutative and isomorphic to $k^{\Gamma}$.  \end{proposition}

\begin{proof}
Since $H$ is not commutative and $H_+$ is, the exact
sequence follows from Proposition \ref{exacta-gene}. The Hopf
algebra $H_{\coad}$ is determined by $H_{\coad}-\comod = (H-\comod)_{\ad}$, and we have
$(H-\comod)_{\ad}=H_+-\comod $ by Proposition \ref{genegrading}.
Hence $H_{\coad} \cong H_+ \cong k^\Gamma$.
\end{proof}

\begin{remark}
Another way to prove the proposition is to use Theorem \ref{shortspecial}.
\end{remark}

Since the extension \eqref{exacta2} is cocentral, the
proposition implies that $H^*$ has a central group-like element of
order 2.

\begin{remark}\label{iso-ext} Consider another quotient Hopf algebra $\A \to H'$, giving rise
to an exact sequence $k \to k^{\Gamma'} \to H' \to k\Z_2 \to k$,
as in Proposition \ref{extension2}. Assume that  $H \simeq H'$ as
Hopf algebras. Then $\Gamma \simeq \Gamma'$ and this exact
sequence is isomorphic to \eqref{exacta2}.

\begin{proof} By construction, we have $k^{\Gamma} = H_{\coad}$ and $k^{\Gamma'} = {H'}_{\coad}$.
Therefore a Hopf algebra isomorphism $H \to H'$ must restrict to
an isomorphism $k^{\Gamma}\to k^{\Gamma'}$. Then it induces an
isomorphism of the corresponding exact sequences. \end{proof}
\end{remark}

As explained in Subsection \ref{exact-seq}, the extension
\eqref{exacta2} is determined by an action of $\Z_2$ on $\Gamma$,
or in other words, by a group automorphism $\theta$ of order $2$,
and a pair $(\sigma, \tau)$ of compatible cocycles.

\smallbreak By Remark \ref{iso-ext}, this data determines the
isomorphism class of $H$ which is that of a bicrossed product
$k^{\Gamma} {}^{\tau}\#_{\sigma}k\Z_2$.

\begin{remark}\label{sigma} Since $\Z_2$ is cyclic, we may apply \cite[Lemma 1.2.5]{pqq} to
the dual extension $k^{\Z_2} {}^{\sigma}\#_{\tau}k\Gamma$, whence
we get that the class of $\sigma$ is trivial in $H^2(\Z_2,
(k^{\Gamma})^{\times})$.

By \cite[Proposition 5.2]{Maext}, after eventually changing the
representative of the class of $\tau$ in $H^2(\Gamma,
(k^{\Z_2})^{\times})$, we may assume that $\sigma = 1$. This can
be seen alternatively, applying \cite[Theorem 4.4]{mastnak}.

The automorphism $\theta$  and the class of $\tau$ will be
explicitly determined in Lemmas \ref{aktion} and \ref{coalgebra},
respectively.
\end{remark}

\medbreak Recall that a finite Hopf algebra over the field of
complex numbers is called a \emph{Kac algebra} if it is a
$C^*$-algebra and all structure maps are $C^*$-algebra maps.

\begin{corollary}\label{kac} Suppose $k = \C$. Then $H$ admits a Kac algebra structure. \end{corollary}


\begin{proof} It follows from Proposition \ref{extension2} and
\cite{ma-ff}. \end{proof}

Let $H$ be as in \eqref{exacta2}. The next corollary describes the
category $\Rep H$ of finite dimensional representations of $H$.
Let $\vect^{\Gamma}$ denote the fusion category of finite
dimensional $\Gamma$-graded vector spaces.

\begin{corollary}\label{representations} There is an action of $\Z_2$ on $\vect^{\Gamma}$ by tensor
autoequivalences such that $\Rep H$ is equivalent to the
$\Z_2$-equivariantization $(\vect^{\Gamma})^{\Z_2}$.
\end{corollary}

\begin{proof} It follows from Proposition \ref{extension2}
and \cite[Proposition 3.5]{extriang}. \end{proof}

\begin{remark}\label{omega} For every finite subgroup $\Gamma$ of $\PSL_2(k)$, as listed in (i)--(v) of Subsection \ref{subgrupos}, there is a
subgroup $\widetilde \Gamma$ (often denoted $2\Gamma$) of
$\SL_2(k)$ such that $\widetilde \Gamma / \{ \pm I \} = \Gamma$.

The groups $\widetilde \Gamma$ are called the \emph{binary
polyhedral groups}. These groups are non-split central extensions
of $\Z_2$ by the corresponding polyhedral group:
\begin{equation}\label{binary}0 \to \Z_2 \to \widetilde
\Gamma \to \Gamma \to 1.
\end{equation} Furthermore, we have $Z(\widetilde \Gamma) \simeq
\Z_2$. A binary dihedral  $2$-group is also called a generalized
quaternion group. There are isomorphisms $\widetilde T \simeq
\SL(2, 3)$, $\widetilde I \simeq \SL(2, 5)$  \cite{coxeter},
\cite[Theorem 6.17]{suzuki}.

\smallbreak We shall denote by $\omega \in H^2(\Gamma, \Z_2)$ the
nontrivial cohomology class arising from the central extension
\eqref{binary}.\end{remark}

\begin{remark}\label{schur-mult} For the finite subgroups of $\PSL_2(k)$ listed before, we have
$H^2(\Z_n, k^{\times}) = 1$, $H^2(D_n, k^{\times}) = 1$, for $n$
odd, $H^2(D_n, k^{\times}) \simeq \Z_2$, for  $n$ even, while the
groups $H^2(T, k^{\times})$, $H^2(O, k^{\times})$ and $H^2(I,
k^{\times})$ are all cyclic of order $2$. See \cite[Proposition
4.6.4 and Theorems 4.8.3, 4.8.5]{karpilovsky}.

\smallbreak In particular, if $\Gamma$ is one of the groups $T$,
$O$, $I$ or $D_n$, $n$ even, then $\widetilde \Gamma$ is a
representation group of $\Gamma$. That is, every (irreducible)
projective representation of $\Gamma$ lifts uniquely to an
(irreducible) linear representation of $\widetilde \Gamma$. See
\cite[Chapter 3]{karpilovsky}. \end{remark}

\subsection{Description}\label{descripcion}
Again $H$ is a nontrivial finite dimensional quotient Hopf algebra
of $\A = \Oo_{-1}[\SL_2(k)]$. We shall next determine the
cohomological data arising from the extension \eqref{exacta2}.

\smallbreak Recall that such an extension is determined by a group
automorphism  $\theta \in \Aut \Gamma$ of order $2$, and an
element of the group $\Opext(k^{\Gamma}, k \Z_2)$ associated to
$\theta$.

\smallbreak Let $\alpha_{ij} = \pi(a_{ij})$, $1\leq i, j\leq 2$,
denote the image in $H$ of the standard generators of $\A$. Thus
the $\alpha_{ij}$'s span the simple subcoalgebra $C$ of $H$ such
that $H = k[C]$.

Let also $p: H \to k\Z_2$ denote the (cocentral) Hopf algebra map
with $k^{\Gamma} \simeq H^{\co p} \subseteq H$, and $p(a_{ij})=\delta_{ij}x$,
where $1\neq x \in \Z_2$.

\begin{lemma}\label{coalgebra}
There is an isomorphism of coalgebras $H \to k^{\widetilde
\Gamma}$ preserving the Hopf subalgebra $k^{\Gamma}$.

Identifying $\Z_2 \simeq \widehat{\Z_2}$, the class of the cocycle
$\tau: \Gamma \times \Gamma \to (k^{\Z_2})^{\times}$ arising from
the exact sequence \eqref{exacta2} is determined by
\begin{equation}\label{formula-tau}\tau(s, t)(p) = p\circ
\omega(s, t),
\end{equation} for all $p \in \widehat{\Z_2}$, $s, t \in \Gamma$, where $\omega \in H^2(\Gamma,
\Z_2)$ is as in Remark \ref{omega}. \end{lemma}

\begin{proof} Since $H$ is nontrivial, the commutative diagram
\eqref{comm-diag} can be extended to a commutative diagram with
exact rows
\begin{equation}\label{largo}\begin{CD}k @>>> \A_+ @>>> A @>>> k\Z_2 @>>> k \\
@VVV @VV{=}V        @VV{\lambda}V @VV{=}V @VVV \\
k @>>> \A_+ @>>> \A @>>> k\Z_2 @>>> k\\
@VVV @VV{\pi}V        @VV{\pi}V @VV{=}V @VVV \\
k @>>> k^{\Gamma} @>>> H @>>> k\Z_2 @>>> k,
\end{CD}\end{equation}
where commutativity of the bottom right square follows from
the construction of the exact sequence in the proof of Proposition \ref{exacta-gene}. In particular, $\pi$ is a $\Z_2$-graded map with
respect to the canonical $\Z_2$-gradings on $\A$ and $H$.

Letting $\J \subseteq \A$ denote the kernel of $\pi$, it follows
that $\J$ is a $\Z_2$-homoge\-neous ideal. Then, by Remark
\ref{correspondencia-lambda}, $\lambda^{-1}(\J) = J$ is an ideal
of $A$. Moreover, $J$ is also a coideal of $A$, since $\lambda$ is
a coalgebra isomorphism.

Since the composition $\pi\lambda: A \to H$ is surjective, this
allows us to identify $H \simeq k^{S}$ as coalgebras, where $S \subseteq
\SL_2(k)$ is a finite subset. By the above,  $\lambda^{-1}(\J)$ is
a Hopf ideal of $A$. Then $k^{S}$ is a quotient Hopf algebra of
$A$, that is, $S$ is a subgroup of $\SL_2(k)$. Commutativity of
the diagram \eqref{largo} implies that $S$ is isomorphic to the
binary polyhedral group associated to $\Gamma$.

\medbreak Since $S \simeq \widetilde \Gamma$, then $S$ fits into
an exact sequence \eqref{binary}, and thus $S$ can be identified
with $\Z_2 \times \Gamma$ with the product given, for $s, t \in
\Gamma$, $x, y \in \Z_2$, by the formula
\begin{equation*} (x, s)(y, t) = (xy\omega_0(s, t), st),
\end{equation*}where $\omega_0 \in Z^2(\Gamma, \Z_2)$ represents
the class $\omega$. Then $k^S$ is identified with $k^{\Z_2\times
\Gamma} \simeq k^{\Gamma} \otimes k^{\Z_2}$ with the (dual)
coproduct:
$$\Delta(e_g \otimes p)(x\otimes s \otimes y \otimes t) = (e_g
\otimes p) (xy\omega_0(s, t)\otimes st) = e_g(st) p (xy\omega_0(s,
t)),$$ for all $s, t, g \in \Gamma$, $x, y \in \Z_2$, $p \in
k^{\Z_2}$. Then, for all $p \in \widehat \Z_2 \subseteq k^{\Z_2}$,
we get
\begin{equation*}\Delta(e_g \otimes p) = \sum_{st = g} (p\circ\omega_0)(s,
t)(e_s \otimes p) \otimes (e_t \otimes p). \end{equation*}
Comparing this expression with formula \eqref{comult}, we obtain
formula \eqref{formula-tau} for $\tau$. This finishes the proof of
the lemma. \end{proof}

\begin{remark} In view of Lemma
\ref{coalgebra}, the Hopf algebras $H$ can be regarded as
'deformations' of the binary polyhedral groups $\widetilde
\Gamma$. Observe that, since every Sylow subgroup of $\widetilde
\Gamma$ is either generalized quaternion or cyclic, then
$\widetilde \Gamma$ does not admit cocycle deformations.

Indeed, by the classification of (dual) cocycle deformations in
\cite{movshev, davydov}, such deformations are classified by pairs
$(S, \alpha)$, where $S$ is a subgroup and $\alpha$ is a
nondegenerate $2$-cocyle on $S$. In particular, the order of $S$
is necessarily a square. For the reason we mentioned, $\widetilde
\Gamma$ contains no nontrivial such subgroup which admits a
nondegenerate $2$-cocycle.
\end{remark}

Let $\rho: \Gamma \to \text{GL}(V)$ be a finite dimensional
projective representation of $\Gamma$ on the vector space $V$.
Recall that a factor set of $\rho$ is a $2$-cocycle $\alpha:
\Gamma \times \Gamma \to k^{\times}$ such that
$$\rho(s) \rho(t) = \alpha(s, t) \, \rho(st),$$ for all
$s, t \in \Gamma$. Under linear equivalence, the classes of
projective representations with a given factor set $\alpha$
correspond to isomorphism classes of representations of the
twisted group algebra $k_{\alpha}\Gamma$.

In particular, if $\rho$, $\rho'$ are projective representations
on the vector space $V$, and $X, X' \subseteq k^{\Gamma}$ are the
subspaces of matrix coefficients of $\rho$ and $\rho'$,
respectively, then $\rho$ is linearly equivalent to $\rho'$ if and
only if $X = X'$.

\medbreak In contrast with the notion of linear equivalence, there
is a notion of projective equivalence. Projectively equivalent
representations may give rise to factor sets differing by a
coboundary. They correspond to isomorphic (linear)
representations, when lifted to the representation group.

\medbreak Let $\tau$ be the $2$-cocycle determined by
\eqref{formula-tau}. That is, $\tau_p \in Z^2(\Gamma, k^{\times})$
is a $2$-cocycle representing the class of $p\circ \omega$, where
$1\neq p \in \widehat \Z_2$. Since $|H^2(\Gamma, k^{\times})| \leq
2$, we may and shall assume in what follows that $\tau_p^{-1} =
\tau_p$.

Suppose $\rho$ is a projective representation of $\Gamma$ with
factor set $\tau_p$, that lifts to a self-dual representation of
the group $\widetilde \Gamma$, that is, such that $\rho^*$ is
projectively equivalent to $\rho$. Since $\tau_p^{-1} = \tau_p$,
then $\rho^*$ is indeed linearly equivalent to $\rho$.

In particular, if $X \subseteq k^{\Gamma}$ is the subspace of
matrix coefficients of $\rho$, then we have $\mathcal S(X) = X$.

\begin{lemma}\label{aktion} Suppose that the quotient Hopf algebra
$\Oo_{-1}[\SL_2(k)] \to H$, where $H \simeq k^{\Gamma}
{}^{\tau}\#k\mathbb Z_2$, affords the automorphism $\theta \in
\Aut (\Gamma)$ of order $2$. Then there exists a $2$-dimensional
faithful irreducible projective representation $\rho$ of $\Gamma$
with factor set $\tau_p$ such that  $\rho \circ \theta$ is
linearly equivalent to $\rho$.

Conversely, given such $\rho$ and $\theta$, there exists a unique
noncommutative semisimple Hopf algebra extension $H =
k^{\Gamma}{}^{\tau}\# k\Z_2$ affording the automorphism $\theta$,
such that $H$ is a quotient of $\Oo_{-1}[\SL_2(k)]$. \end{lemma}

\begin{proof} We shall identify $\Z_2 \simeq \widehat \Z_2$ in what follows.
Let $H \simeq k^{\Gamma} {}^{\tau}\#k\mathbb Z_2$, affording the
automorphism $\theta \in \Aut (\Gamma)$ of order $2$. Since $H$ is
a quotient $\Oo_{-1}[\SL_2(k)] \to H$, there exists
a $4$-dimensional simple subcoalgebra $C\subseteq H$ such that
$k[C] = H$ and $\mathcal S(C) = C$.

In view of Remark \ref{coalg-abel}, $C = X \# p$, where $X$ is
identified with a $4$-dimensional simple subcoalgebra of the
twisted group algebra $(k_{\tau_p}\Gamma)^*$. Let $\rho$ be the
irreducible projective representation corresponding to $X$. In
view of Lemma \ref{coalgebra}, under the canonical coalgebra
isomorphism $H \to k^{\widetilde \Gamma}$, $C$ corresponds to a
self-dual faithful irreducible representation of $\widetilde
\Gamma$ that lifts the projective representation $\rho$. Thus
$\rho$ is a faithful self-dual projective representation.
Moreover, since $\tau_p = \tau_p^{-1}$, then $\rho^*$ is linearly
equivalent to $\rho$; that is, $\mathcal S(X) = X$ in
$k^{\Gamma}$.

Recall that the action of $p \in \Z_2$ on $\Gamma$ is given by the
automorphism $\theta$. In view of formula \eqref{antipode} for the
antipode of $H$, condition $\mathcal S(C) = C$ implies that that
$p . X = \mathcal S(X) = X$, where $.:k\Z_2 \otimes k^{\Gamma} \to
k^{\Gamma}$ is the action transpose to $\vartriangleleft$. This
amounts to the condition that $\rho \circ \theta$ be linearly
equivalent to $\rho$.

Indeed, $p.X$ is the span of matrix coefficients of the projective
representation $\rho \circ \theta$ of $\Gamma$. On the other hand,
if $c = x\# p \in C$, then, in view of formula \eqref{antipode}
for the antipode of $H$, we have
\begin{align*}\mathcal S(c) & =
\sum_s \tau_p(s^{-1}, s)^{-1} x(s) \; e_{(s\vartriangleleft p)^{-1}}\# p \\
& = \sum_s \tau_p((s\vartriangleleft p)^{-1}, s\vartriangleleft
p)^{-1} x(s\vartriangleleft p) \; e_{s^{-1}}\# p,
\end{align*} and the last expression belongs to $p.\mathcal S(X)\# p = p.X \# p$.
Hence $\mathcal S(C) = C$, if and only if, $p.X = X$, if and only
if, $\rho \circ \theta$ is linearly equivalent to $\theta$.

\medbreak Conversely, suppose given such $\theta$ and $\rho$. Let
$1\neq p \in \widehat \Z_2$. We have in particular,
$\theta^*(\tau_p) = \tau_p^{-1}$. This implies that $(1, \tau)$ is
a pair of compatible cocycles, thus giving rise to a Hopf algebra
$H = k^{\Gamma} {}^{\tau}\# k\Z_2$.

Let $X\subseteq k^{\Gamma}$ be the span of the matrix coefficients
of the  $2$-dimensional self-dual faithful irreducible projective
representation $\rho$, and let $C = X \# p \subseteq  H$.

By assumption, $\rho \circ \theta$ is linearly equivalent to
$\rho$. Since the action of $p \in \Z_2$ on $\Gamma$ is given by
the automorphism $\theta$, this implies that $p . X = X$, where
$.:k\Z_2 \otimes k^{\Gamma} \to k^{\Gamma}$ is the action
transpose to $\vartriangleleft$. Hence, as before, formula
\eqref{antipode} for the antipode of $H$ implies that $\mathcal
S(C) = C$. Similarly, formula \eqref{mult} for the multiplication
implies that $k[C] = H$. Since $H$ is not commutative, then $H$
must be a quotient of $\Oo_{-1}[\SL_2(k)]$. This finishes the
proof of the lemma.
\end{proof}

Let us recall from \cite{ma-contemp} some properties of the
nontrivial Hopf algebras $\mathcal A_{4n}$ and $\mathcal B_{4n}$
constructed there.

Consider the presentation of the dihedral group $D_n$ by
generators $s_+$ and $s_-$ satisfying the relations $$s_{\pm}^2 =
(s_+s_-)^n = 1.$$
Both $H = \mathcal A_{4n}$ and $H = \mathcal B_{4n}$ are
extensions $k \to k^{\Z_2} \to H \to kD_n \to k$ associated to the
matched pair $(k^{\Z_2},kD_n)$, where the action of $\Z_2$ on
$D_n$ is given by
\begin{equation}\label{mas-action}a \vartriangleright s_{\pm} =
s_{\mp}. \end{equation}

We shall use the notation $\mathcal A[\widetilde D_n] : = \mathcal
A_{4n}$ and $\mathcal B[\widetilde D_n] : = \mathcal B_{4n}$. Then
$\mathcal A[\widetilde D_n]$ and $\mathcal B[\widetilde D_n]$ are
nonequivalent representatives of classes in
$\Opext(k^{\Z_2},kD_n)$ associated to the action
\eqref{mas-action}. Moreover, every extension arising from this
matched pair is isomorphic to one of $\mathcal A[\widetilde D_n]$
or $\mathcal B[\widetilde D_n]$.

\medbreak Suppose that $\Gamma$ is one of the groups $T \simeq
\mathbb A_4$, $O \simeq \mathbb S_4$, or $I \simeq \mathbb A_5$.

In the first two cases we have $\Aut \Gamma \simeq \mathbb S_4$,
while $\Aut \mathbb A_5 \simeq \mathbb S_5$. In all cases, the
automorphims  $\theta$ of order $2$ of $\Gamma$ are induced by the
adjoint action of a transposition $(.\, .)$, or a product of two
disjoint transpositions $(.\, .) (.\, .)$, viewed as elements in
the corresponding symmetric group.

\medbreak Consider the case where $\Gamma = \mathbb A_4$. In this
case, $k_{\tau_p}\Gamma$ has exactly $3$ irreducible
representations of degree $2$, and only one of them is self-dual:
this can be seen from the fact that $\widetilde{\mathbb A_4}$ is a
representation group of $\mathbb A_4$. See \cite[Table
A.12]{stekolshchik} for the character table of $\widetilde{\mathbb
A_4}$, where the self-dual representation is the one with
character $\chi_3$.

Then, in this case, $\rho$ corresponds to $\chi_3$ and it is
necessarily stable under all automorphisms $\theta$ of $\mathbb
A_4$.

\medbreak Similarly, when $\Gamma = \mathbb S_4$, every
automorphism is inner, and therefore stabilizes all projective
representations.

By Lemma \ref{aktion}, in the cases  $\Gamma = \mathbb A_4$ or
$\mathbb S_4$, we have two nonisomorphic associated Hopf algebras
with a faithful self-dual comodule of dimension $2$. We shall
denote the Hopf algebra $H$ associated to $\theta$ by $\mathcal
A[\widetilde \Gamma]$ or $\mathcal B[\widetilde \Gamma]$, if
$\theta$ corresponds to a transposition or to a product of two
disjoint transpositions, respectively.

\medbreak Now assume that $\Gamma = \mathbb A_5$. It follows from
inspection of the character table of $\widetilde{\mathbb A_5}$
\cite[Table A.19]{stekolshchik} that $k_{\tau_p}\mathbb A_5$ has
$2$ nonequivalent irreducible representations of degree $2$, one
of degree $4$ and one of degree $6$. Moreover, we have

\begin{lemma} The nonequivalent  projective representations of degree $2$ are
conjugated by an outer automorphism of $\mathbb A_5$ induced by a
transposition in $\mathbb S_5$. \end{lemma}

\begin{proof} Consider the presentation of $\widetilde{\mathbb
A_5}$ by generators $a$ and $b$ and relations $a^5 = b^3 = (ab)^2
= -1$. There are in $\widetilde{\mathbb A_5}$ nine classes under
conjugation represented, respectively, by $1$, $-1$, $a$, $a^2$,
$a^3$, $a^4$, $b$, $b^2$ and $ab$. According to \cite[Table
A.19]{stekolshchik}, $\widetilde{\mathbb A_5}$ has two
nonequivalent irreducible representations of degree $2$, whose
characters $\chi$ and $\chi'$ are determined, respectively, by the
following table:
$$\begin{tabular}{|l|l|l|l|l|l|l|l|l|l|} \hline
{\quad } & {$1$ \;} & {$-1$ \;} & {$a$ \;} & {$a^2$ \;} & {$a^3$ \;} & {$a^4$ \;} & {$b$ \;} & {$b^2$ \;} & {$ab$ \;} \\
\hline $\chi$ & $2$ & $-2$ & $\varphi^+$ & $-\varphi^-$ &
$\varphi^-$ & $-\varphi^+$ & $1$ & $-1$ & $0$ \\ \hline $\chi'$ &
$2$ & $-2$ & $\varphi^-$ & $-\varphi^+$ & $\varphi^+$ &
$-\varphi^-$ & $1$ & $-1$ & $0$ \\ \hline \end{tabular}$$ where
$\varphi^{\pm} = \frac{1\pm \sqrt{5}}{2}$.

Let $\rho, \rho': \widetilde{\mathbb A_5} \to \SL_2(k)$ be the
irreducible representations with characters $\chi$, $\chi'$.
Consider the surjective group homomorphism $\pi:
\widetilde{\mathbb A_5} \to \mathbb A_5$ with kernel $\{ \pm 1\}$,
such that $\pi(a) = (12345)$ and $\pi(b) = (153)$.

Letting $\overline s$ denote the element $\pi(s)$, $s \in
\widetilde{\mathbb A_5}$, the classes of $\mathbb A_5$ under
conjugation are represented by $\overline 1 = 1$, $\overline a =
(12345)$, $\overline{a^3}$, $\overline b$ and $\overline{ab}$. We
have also that the class of $\overline{a^3}$ coincides with the
class of $(21345)$.

After composing with a suitable section $\mathbb A_5 \to
\widetilde{\mathbb A_5}$ of $\pi$, $\rho$ and $\rho'$ define
irreducible projective representations $\overline \rho$ and
$\overline \rho'$ of $\mathbb A_5$, respectively. The projective
characters are determined by the following table:
$$\begin{tabular}{|l|l|l|l|l|l|} \hline
{\quad } & {$1$ \;} &  {$(12345)$ \;} & {$(21345)$ \;} &  {$(153)$ \;} &  {$(12)(34)$ \;} \\
\hline $\overline\chi$ & $2$ &  $\varphi^+$ & $\varphi^-$ & $1$ &
$0$ \\ \hline $\overline \chi'$ & $2$ &  $\varphi^-$ &
$\varphi^+$ & $1$ & $0$ \\ \hline \end{tabular}$$ In particular,
$\overline \rho$ and $\overline \rho'$ are not equivalent.
Furthermore, if $\theta$ is the automorphism determined by the
adjoint action of the transposition $(12)$, then we see from the
table above, that $\overline \chi' = \overline\chi \circ \theta$.
This implies the statement on the projective representations in
view of \cite[Theorem 7.1.11]{karpilovsky}. \end{proof}

In view of the above lemma and Lemma \ref{aktion}, when the
automorphism $\theta$ corresponds to a transposition in $\mathbb
S_5$, the bicrossed product associated to $\theta$ has no
irreducible self-dual comodule of dimension $2$. When $\theta$
corresponds to a product of two disjoint transpositions, the
automorphism is inner and we have an associated Hopf algebra that
we shall denote $\mathcal B[\widetilde{\mathbb A_5}]$, as before.

\begin{theorem}\label{explicito} Let $H$ be a nontrivial Hopf algebra quotient of
$\Oo_{-1}[\SL_2(k)]$. Then $H$ is isomorphic to exactly one of the
Hopf algebras $\mathcal B[\widetilde I]$,  $\mathcal A[\widetilde
\Gamma]$ or $\mathcal B[\widetilde \Gamma]$, where $I \simeq
\mathbb A_5$, and $\Gamma$ is one of the groups $D_n$, $n \geq 2$,
$T \simeq \mathbb A_4$, $O \simeq \mathbb S_4$.
\end{theorem}

\begin{proof} By Proposition \ref{extension2} and Theorem \ref{explicito},
since  $H$ is nontrivial, then $H$ is an extension
\eqref{exacta2}. Hence $H$ corresponds to the  class of some pair
$(\sigma, \tau)$ in $\Opext(k^{\Gamma}, k\Z_2)$. By Lemma
\ref{coalgebra}, we know that $\tau_p$ represents the class $p
\circ \omega$. Moreover, by \cite[Lemma 1.2.5]{pqq}, \cite[Theorem
4.4]{mastnak}, we may assume that $\sigma = 1$. Moreover, by
Remark \ref{iso-ext}, two quotients are isomorphic if and only if
the exact sequences are isomorphic.

Consider the case where $\Gamma$ is one of the groups $T$, $O$ or
$I$. In view of Lemma \ref{aktion} and the previous discussion, it
remains to determine which possible automorphisms $\theta$,
$\theta'$, give isomorphic Hopf algebras $H$, $H'$. Since a Hopf
algebra isomorphism $H \to H'$ is an isomorphism of extensions,
this is the case if and only if the associated matched pairs are
isomorphic, which amounts to the automorphisms $\theta$ and
$\theta'$ being conjugated in $\Aut(\Gamma)$.

\medbreak The case where $\Gamma = D_n$, $n \geq 2$, will be
treated separately in Lemma \ref{dn}. This will imply, as a
byproduct, a comparison with the nontrivial examples studied by
Masuoka in \cite{ma-contemp}. \end{proof}

The following facts about fusion rules for the irreducible
characters of $k^{D_n}$, $n \geq 2$, will be used in the proof of
our next lemma.

If $n$ is odd, $k^{D_n}$ has two irreducible characters $1$, $a$,
of degree $1$, and $r = (n-1)/2$ irreducible characters of degree
$2$, $\chi_1, \dots, \chi_r$, such that
\begin{equation}\label{nodd}a\chi_i = \chi_i, \quad \chi_i\chi_j =
\begin{cases}\chi_{i+j} + \chi_{|i-j|}, \; i+j \leq r,\\ \chi_{n-(i+j)} + \chi_{|i-j|}, \; i+j > r,
\end{cases}\end{equation}
where $\chi_0 = 1 + a$.

On the other hand, if $n$ is even, $k^{D_n}$ has four irreducible
characters $1$, $a$, $b$, $c = ab$, of degree $1$, and $r = n/2-1$
irreducible characters of degree $2$, $\chi_1, \dots, \chi_r$,
such that \begin{equation}\label{neven}a\chi_i  = \chi_i, \quad
b\chi_i = \chi_{n/2-i}, \quad \chi_i\chi_j =
\begin{cases}\chi_{i+j} + \chi_{|i-j|}, \; i+j \leq n/2,\\ \chi_{n-(i+j)} + \chi_{|i-j|}, \; i+j > n/2,
\end{cases}\end{equation}
where $\chi_0 = 1 + a$, $\chi_{n/2} = b+c$.

\begin{lemma}\label{dn} Let $H$ be a nontrivial Hopf algebra quotient of
$\Oo_{-1}[\SL_2(k)]$ such that $H$ fits into an exact sequence
\eqref{exacta2}, with $\Gamma = D_n$, $n \geq 2$. Then $H$ is
isomorphic to one of the Hopf algebras $\mathcal A[\widetilde
D_n]$ or $\mathcal B[\widetilde D_n]$. \end{lemma}

By \cite[Proposition 3.13]{ma-contemp} the Hopf algebras $\mathcal
A[\widetilde D_n]$ and $\mathcal B[\widetilde D_n]$ are not
isomorphic, if $n \geq 3$.

Note that the fusion rules for $\mathcal A[\widetilde D_n]$ or
$\mathcal B[\widetilde D_n]$ are the same as those for
$k^{D_{2n}}$, as given by \eqref{neven} \cite[Proposition
3.9]{ma-contemp}. Then $\mathcal A[\widetilde D_n]$, $\mathcal
B[\widetilde D_n]$ do satisfy the assumptions in the lemma.

\begin{proof} In view of previous classification results, we shall
assume in the proof that $n \neq 2$. That is, we shall not
consider the case $\dim H = 8$, where the result is well-known.

Formula \eqref{comult} for the comultiplication in $H$ implies
that $H = (k\Gamma)^* \oplus (k_{\tau_p}\Gamma)^*$, where
$k_{\tau_p}\Gamma$ is the twisted group algebra, and $1 \neq p \in
\Z_2$. In particular, every irreducible $H$-comodule has dimension
$1$ or $2$. See \cite[Theorem 3.7.3]{karpilovsky}.

Let $\chi \in C$ be the irreducible character of $C$. Then,
decomposing $\chi\chi^* = \chi^2$ into a sum of irreducible
characters, we get  $|G[\chi]| = 2$. Indeed, the assumption that
$k[C] = H$ together with $\dim H > 8$, imply that $|G[\chi]| \neq
4$.

\medbreak Let $G[\chi] = \{1, a\}$, where $a^2 = 1$. Then $\chi^2
= 1 + a + \lambda$, where $\lambda$ is an irreducible character of
degree $2$, such that $a \in G[\lambda]$. Thus $k[\lambda] =
H_{\coad} \simeq k^{D_n}$. In particular, $kG[\chi]$ is normal in
$H_{\coad}$. We claim that $kG[\chi]$ is a central Hopf subalgebra
of $H$.

Since $a\lambda = \lambda = \lambda a$, and the simple
subcoalgebra containing $\lambda$ generates $H_{\coad}$, then the
quotient Hopf algebra $H_{\coad}/H_{\coad}(kG[\chi])^+$ is
cocommutative, by \cite[Remark 3.2.7 and Corollary 3.3.2]{ssld}.

This implies that $a\mu = \mu = \mu a$, for all irreducible
character $\mu \in H_{\coad}$ of degree $2$. By assumption, $\dim
H > 8$.  and taking into account the fusion rules for the
irreducible characters in $k^{D_n}$ in \eqref{nodd} and
\eqref{neven}, we may assume that $a\in k^{D_n}$ is the only
nontrivial group-like element with this property.

Indeed, otherwise, we would have $n$ even and $b\chi_1 = \chi_1$,
implying that $n = 4$. Then $\dim H = 16$ and $|G(H)| = 4$, thus
the lemma follows in this case from \cite{kashina}.

\medbreak Hence $a$ must be stable under the action of $\Z_2$
coming from the sequence \eqref{exacta2}, because this is an
action by Hopf algebra automorphisms. Hence $kG[\chi]$ is a normal
(therefore central) Hopf subalgebra of $H$, as claimed.

\medbreak Moreover, since $aC = C = Ca$, and $C$ generates $H$,
then as before $H/H(kG[\chi])^+$ is cocommutative. Therefore we
get an exact sequence \begin{equation}\label{unamas}k \to k^{\Z_2}
\to H \to kF \to k,\end{equation} where $F$ is a group of order
$2n$.

We shall show next that $F \simeq D_n$ and the action $\Z_2 \times
F \to F$ associated to the exact sequence \eqref{unamas} coincides
with \eqref{mas-action}. This will imply the lemma, in view of
\cite[Proposition 3.11]{ma-contemp}.

\medbreak Denote by $\pi: H \to kF$ the projection in
\eqref{unamas}. We may decompose $\pi(\chi) = x + y$, where $x, y
\in F$. Moreover, since $\chi$ is a faithful character, then $x$
and $y$ generate $F$.

Since $\chi^* = \chi$, then either $x^2 = y^2 = 1$ or $x =
y^{-1}$. If the last possibility holds, then $F = \langle x, y
\rangle$ would be cyclic, thus implying that $H$ is commutative,
against the assumption. Therefore $x^2 = y^2 = 1$ in $F$.

On the other hand, we have $\chi^2 = 1 + a + \lambda$, and
$\pi(\chi^2) = x^2 + y^2 + xy + yx = 2. 1 + xy + (xy)^{-1}$. Hence
$\pi(\lambda) = xy + (xy)^{-1}$. But the components of
$\pi(\lambda)$ generate a subgroup of index $2$ of $F$ whose group
algebra is isomorphic to $\pi(H_{\coad}) = \pi(k^{D_n})$. Hence
$\langle xy \rangle \simeq \Z_n$. Thus we obtain the relation
$(xy)^n = 1$. Hence $F \simeq D_n$.

\medbreak As in the proof in \cite[pp. 209]{ma-contemp}, the
simple subcoalgebra $C$ is a left $k\langle a \rangle$-submodule,
and its image in $kF$ is spanned by the group-likes $x$ and $y$.
Then $C$ is a crossed product $C = k^{\Z_2}\# k\{x, y\}$, and
therefore the set $\{x, y\}$ is an orbit under the action of
$\Z_2$. Letting $s_+ = x$, $s_- = y$, we recover the action
\eqref{mas-action}. This finishes the proof of the lemma.
\end{proof}

\medbreak The following lemma gives the degrees of the irreducible
representations  of $H$.

\begin{lemma}\label{irrep-h}  Let $p$ be the order of the subgroup
$\Gamma^{\Z_2}$ of fixed points of $\Gamma$ under the action of
$\Z_2$. Then $H$ has $2p$ irreducible representations of degree
$1$ and the remaining ones are of degree $2$. In particular,
$|G(H^*)| \geq 4$.
\end{lemma}


\begin{proof} As an algebra, $H$ is a smash
product $k^{\Gamma}\#k\Z_2$. Then the irreducible representations
of $H$ are of the form $W_s : = \Ind_{k^{\Gamma}\otimes k C_s}^H
s\otimes U_s$, where $s$ runs over a set of representatives of the
orbits of $g$ in $\Gamma$, $C_s \subseteq \Z_2$ is the stabilizer
of $s$, and $U_s$ is an irreducible representation of $C_s$ \cite{MoW}.

In particular, we have $\dim W_s = [\Z_2: C_s]$. Hence $\dim W_s =
1$ if and only if $s \in \Gamma^{\Z_2}$. Note in addition that
such an automorphism $g$ must have a nontrivial fixed point
(otherwise, $\Gamma$ would be abelian of odd order); whence the
claimed inequality for $|G(H^*)|$.
\end{proof}

\begin{remark}\label{chardeg} Let $H  = \mathcal A[\widetilde \Gamma]$
or $\mathcal B[\widetilde \Gamma]$, be one of the nontrivial Hopf
algebras corresponding to $\widetilde \Gamma$.

Let $p \geq 2$ be the order of the subgroup of invariants  in
$\Gamma$ under the action of $\Z_2$, and let $q = (\dim H -
2p)/4$. Then, as an algebra, $H$ is of type $(1,  2p; 2, q)$. See
Lemma \ref{irrep-h}.
Explicitly, we have the following algebra
types, for each possible isomorphism class:
\begin{align*}
& \mathcal A[\widetilde D_{n}]: (1, 4; 2, n-1), \;\; \mathcal B[\widetilde D_{n}]: (1, 4; 2, n-1), \quad (n \geq 2)\\
& \mathcal A[\widetilde T]: (1, 4; 2, 5), \quad \qquad \mathcal B[\widetilde T]: (1, 8; 2, 4),\\
& \mathcal A[\widetilde O]: (1, 8; 2, 10), \;\, \qquad \mathcal B[\widetilde O]:  (1, 16; 2, 8), \\
& \mathcal B[\widetilde I]: (1, 8; 2, 28). \end{align*} On the
other hand, by Lemma \ref{coalgebra}, $H \simeq k^{\widetilde
\Gamma}$ as coalgebras. Hence the coalgebra types are the
following:
\begin{align*}
& \mathcal A[\widetilde D_{n}], \mathcal B[\widetilde D_{n}], \; n \geq 2: \, (1, 4; 2, n-1),\\
& \mathcal A[\widetilde T], \mathcal B[\widetilde T]: \;  (1, 3; 2, 3; 3, 1),\\
& \mathcal A[\widetilde O], \mathcal B[\widetilde O]: \; (1, 2; 2, 3; 3, 2; 4, 1), \\
& \mathcal B[\widetilde I]: \; (1, 1; 2, 2; 3, 2; 4, 2; 5, 1; 6,
1).
\end{align*}
See for instance \cite[A.5]{stekolshchik}. \end{remark}

\begin{remark}\label{fusion}
Let $H$ be one of the Hopf algebra $\mathcal A[\widetilde{\Gamma}]$
or $\mathcal B[\widetilde{\Gamma}]$. Then the fusion rules of $H$ are the same as the ones of $\widetilde{\Gamma}$.
\end{remark}

\begin{proof}
The result is known if $\widetilde{\Gamma}=\widetilde{D_n}$, see \cite{ma-contemp}.
Assume now that $\widetilde{\Gamma} =\widetilde{T}$, $\widetilde{O}$
or $\widetilde{I}$. It is enough, to describe the fusion rules of $H$, to
describe the product of irreducible characters in $H$. These are, using the previous notation, of  the form $\chi \# 1$, $\chi'\#p$, where $\chi$ is an ordinary irreducible character of $\Gamma$ and  $\chi'$ is an irreducible projective
character of $\Gamma$. In view of the product formula \ref{mult}, we need
to describe the action of $\theta$ on (projective) characters.
If $\theta$ is inner (for $\mathcal B[\widetilde{T}]$, $\mathcal A[\widetilde{O}]$, $\mathcal B[\widetilde{O}]$, $\mathcal B[\widetilde{I}]$) this action is trivial and hence the fusion rules are unchanged.

If $ H = \mathcal A[\widetilde{T}]$, the automorphism $\theta$ permutes
the characters $\chi_2-\chi_3$ and $\chi_4-\chi_5$ in Table A.12 of \cite{stekolshchik}.
However the computation of the fusion rules of $\widetilde{T}$ show that after these
permutations, the fusion rules of $\mathcal A[\widetilde{\Gamma}]$ remain the same.
\end{proof}

\begin{proof}[Proof of Theorem \ref{clasificacion}]
Let $H$ be a semisimple Hopf algebra containing a simple
subcoalgebra $C$ of dimension $4$ such that $\mathcal S(C) = C$.
Let $k[C] \subseteq H$ be the Hopf subalgebra generated by $C$.

By Proposition \ref{moti} and Corollary \ref{sym-ssym}, there is a surjective Hopf algebra map $\mathcal H \to k[C]$, where
$\mathcal H = \Oo[\SL_2(k)]$, if $\nu(V) = -1$, and $\mathcal H =
\Oo_{-1}[\SL_2(k)]$, if $\nu(V) = 1$.
Then the proof of the theorem follows from Theorem \ref{explicito}. \end{proof}



\section{Applications}\label{appl}

In this section we shall give some applications of the main
results in previous sections to the classification of semisimple
Hopf algebras.

\subsection{Semisimple Hopf algebras with character degrees at most $2$}
Recall that if  $C$ is a simple subcoalgebra of a semisimple Hopf
algebra, and $\chi \in C$ is the irreducible character contained
in $C$, then $G[\chi]$ is the subgroup of the group-like elements
$g$ such that $g\chi = \chi$, or equivalently, such that $g$
appears with positive multiplicity (necessarily equal to $1$) in
the product $\chi \chi^*$. We shall also denote $B[\chi] =
k[C\mathcal S(C)]$.

\medbreak We give for completeness the proof of the following
lemma that will be used later on. See \cite[Lemma 2.4.1]{ssld}.

\begin{lemma}\label{prod-irred} Let $H$ be a semisimple Hopf algebra and let $\chi, \chi' \in H$ be irreducible characters. Then the
following are equivalent:
\begin{enumerate}\item[(i)] The product $\chi^* \chi'$ is
irreducible.
\item[(ii)] For all irreducible character $\lambda \neq 1$, $m(\lambda, \chi\chi^*) =
0$ or $m(\lambda, \chi'(\chi')^*) = 0$.
\end{enumerate} \end{lemma}
Here $m(\lambda, \chi\chi^*)$, $m(\lambda, \chi'(\chi')^*)$,
denote the multiplicity of $\lambda$ in $\chi\chi^*$ and in
$\chi'(\chi')^*$, respectively.

\begin{proof} Let $\zeta = \chi^* \chi'$. Then $\zeta$
is irreducible if and only if $m(1, \zeta \zeta^*) = 1$. On the
other hand, $$\zeta \zeta^*  = \chi^* \chi' (\chi')^* \chi
= \chi^* \chi +  \sum_{\mu \neq 1} m(\mu,
\chi' (\chi')^*) \chi^* \mu \chi. $$ Therefore, $m (1, \zeta
\zeta^*) = 1$ if and only if  for all $\mu \neq 1$, with $m(\mu,
\chi' (\chi')^*) > 0$, we have $m(1, \chi^* \mu \chi) = 0$ or
equivalently, $m(\mu, \chi \chi^*) = 0$. \end{proof}

\medbreak Let $H$ be a semisimple Hopf algebra such that $\deg
\chi \leq 2$, for all irreducible character $\chi \in H$. Let
$B\subseteq H$ denote the adjoint Hopf subalgebra $B : =
H_{\coad}$. Hence $B = k[B[\chi]:\, \deg\chi = 2]$ is generated as
an algebra by the Hopf subalgebras $B[\chi]$.

By Theorem \ref{nr-ref} $B[\chi]$ is commutative and, in view of
the coalgebra structure of $H$, isomorphic to $k^{\Z_2\times
\Z_2}$ or $k^{D_n}$, $n \geq 3$, for all irreducible character
$\chi$ of degree $2$. In addition, $G[\chi] \neq 1$, for all such
$\chi$.

\begin{lemma}\label{interseccion}  Suppose  $\chi$ is
an irreducible character of $H$ of degree $2$, such that
$|G[\chi]| = 2$. Then $G[\chi] \subseteq \bigcap_{\deg \chi' =
2}B[\chi'] \subseteq Z(B)$.\end{lemma}

\begin{proof} Since $\bigcap_{\deg \chi' =
2}B[\chi'] \subseteq B[\chi']$, for all $\chi'$ of degree $2$,
then it is a central Hopf subalgebra of $B$, because the
$B[\chi']$'s  are commutative and generate $B$ as an algebra. Thus
it will be enough to show that $G[\chi] \subseteq \bigcap_{\deg
\chi' = 2}B[\chi']$.

Let $\chi' \in H$ be any irreducible character of degree $2$. Then
the product $\chi^*\chi'$ cannot be irreducible, because $H$ has
no irreducible character of degree $4$. In view of Lemma
\ref{prod-irred}, there is an irreducible character $\lambda \neq
1$ that appears with positive multiplicity both in $\chi\chi^*$
and in $\chi'(\chi')^*$. This means, in particular, that the Hopf
subalgebra $k[\lambda]$, generated by the simple subcoalgebra
containing $\lambda$, is contained in $B[\chi]\cap B[\chi']$.

If $\deg(\lambda) = 1$, then $1\neq \lambda \in G[\chi] \cap
G[\chi']$. Since $G[\chi]$ is of order $2$ by assumption, this
implies that $G[\chi] \subseteq G[\chi'] \subseteq B[\chi']$.

If $\deg(\lambda) > 1$, then $\lambda$ is irreducible of degree
$2$. In this case we have decompositions $\chi\chi^* = 1 + g +
\lambda$ and $\chi'(\chi')^* = 1 + g' + \lambda$, where $G[\chi] =
\{1, g\}$ and $G[\chi'] = \{1, g'\}$.

Since $g\chi = \chi$, then also $g\lambda = \lambda$, that is,
$G[\chi] \subseteq G[\lambda]$. Therefore $G[\chi] \subseteq
G[\lambda] \subseteq k[\lambda] \subseteq B[\chi']$.  \end{proof}

\begin{corollary}\label{grado2} Let $H$ be a semisimple Hopf algebra such that $\deg \chi \leq
2$, for all irreducible character $\chi \in H$. Let $B =
H_{\coad}$ be the adjoint Hopf subalgebra. Then one of the
following conditions holds:

\begin{enumerate}\item[(i)] $B$ has a central Hopf subalgebra isomorphic to
$k\Gamma$, where $$\Gamma \simeq \underset{m\text{
times}}{\Z_2\times \dots \times \Z_2},$$ such that
$B/B(k\Gamma)^+$ is cocommutative, where $1\leq m \leq n$, and $n$
is the number of irreducible characters $\chi$ of $H$ such that
$|G[\chi]| = 2$,  or
\item[(ii)] $B$ is cocommutative.\end{enumerate}
Further, if \emph{(ii)} holds and $H$ is not cocommutative, then
$4$ divides $\dim B$ and $G(B)$ is generated by elements of order
$2$. \end{corollary}

\begin{proof} We may assume $H$ is not cocommutative.
If $H$ has an irreducible character $\chi$ of degree $2$ such that
$|G[\chi]| = 2$, then, by Lemma \ref{interseccion}, $G[\chi]
\subseteq Z(B)$.  Let $\Gamma$ be the subgroup generated by the
groups $G[\chi]$ such that $|G[\chi]| = 2$. Thus the group algebra
of $\Gamma$  is a Hopf subalgebra contained in the center of $B$.

Further, it follows also from Lemma \ref{interseccion} that
$\Gamma \cap B[\chi] = \Gamma \cap G[\chi] \neq 1$, for all
irreducible character $\chi$ with $|G[\chi]| = 4$. Then,
\cite[Remark 3.2.7 and Corollary 3.3.2]{ssld}, the quotient Hopf
algebra $B/B(k\Gamma)^+$ is cocommutative, and (i) holds in this
case.

\medbreak Otherwise, $|G[\chi]| = 4$, for all irreducible
character $\chi$ of degree $2$. Since $kG[\chi] \subseteq B$, for
all $\chi$, then $4$ divides $\dim B$. On the other hand, by
dimension, $\chi\chi^* = \sum_{g \in G[\chi]}g\in kG(H)$, for all
irreducible character $\chi$ of degree $2$. Then $B \subseteq
kG(H)$ is cocommutative.

Moreover, by Theorem \ref{nr-ref}, we have in this case $G[\chi]
\simeq \Z_2 \times \Z_2$, for all irreducible $\chi$ of degree
$2$. Then (ii) holds, since these subgroups generate $G(B)$.
\end{proof}

\begin{theorem}\label{1m2n} Let $H$ be a nontrivial semisimple Hopf algebra such that $\deg \chi \leq
2$, for all irreducible character $\chi \in H$. Then $H$ is not
simple. More precisely, let $\Gamma \simeq \Z_2\times \dots \times
\Z_2$ be the subgroup of $G(H)$ in Corollary \ref{grado2} (i).
Then one of the following possibilities holds:

\begin{enumerate}\item[(i)] There is a central abelian exact sequence $k \to k\Gamma \to H \to
kF \to k$, where $F$ is a  finite group, or
\item[(ii)] There is a cocentral exact sequence $k \to B \to H \to
kU(\Cc) \to k$, where $B = H_{\coad} \subsetneq H$, and  $U(\Cc)$
is the universal grading group of $\Cc = H-\comod$.
\end{enumerate}
Moreover, in case (ii), either $B$ is commutative or it contains a
central Hopf subalgebra isomorphic to $k\Gamma$.

In particular, either $H$ or $H^*$ contains a nontrivial central
group-like element.
\end{theorem}

\begin{proof} Suppose first that $H_{\coad} = B \subsetneq H$. Then
(ii) holds in view of the properties of the adjoint Hopf
subalgebra $B$. See Subsection \ref{adjoint}.

If, on the other hand, $B = H$, then by Corollary \ref{grado2} we
have (i). \end{proof}

Let $H$ be a semisimple Hopf algebra. Recall from \cite{MoW} that
$H$ is called \emph{lower semisolvable} if there exists a chain of
Hopf subalgebras $H_{n+1} = k \subseteq H_n \subseteq \dots
\subseteq H_1 = H$ such that $H_{i+1}$ is a normal Hopf subalgebra
of $H_i$, for all $i$, and all \emph{factors} ${\overline H}_i : =
H_{i+1}/H_{i+1}H_i^+$ are commutative or cocommutative.

Dually, $H$ is called \emph{upper semisolvable} if there exists a
chain of quotient Hopf algebras $H_{(0)} = H \to H_{(1)} \to \dots
\to H_{(n)}= k$ such that each of the maps $H_{(i-1)} \to H_{(i)}$
is normal, and all \emph{factors} $H_i : = H_{(i-1)}^{{\rm co }
\pi_i}$ are commutative or cocommutative.

We have that $H$ is upper semisolvable if and only if $H^*$ is
lower semisolvable \cite{MoW}.

\begin{remark}\label{ss-ext} Suppose that $H$ fits into an exact sequence $k \to K \to H
\to \overline H \to k$. If $K$ is lower semisolvable and
$\overline H$ is commutative or cocommutative, then $H$ is lower
semisolvable. On the other hand, if $K$ is commutative or
cocommutative and $\overline H$ is upper semisolvable, then $H$ is
upper semisolvable. \end{remark}

\begin{corollary}\label{semisoluble} Let $H$ be a semisimple Hopf algebra such that $\deg \chi \leq
2$, for all irreducible character $\chi \in H$. Then $H$ is lower
semisolvable. \end{corollary}

\begin{proof} We may assume $H$ is not trivial. Then one of the
possibilities (i) or (ii) in Theorem \ref{1m2n} holds. If (i)
holds, then $H$ is both upper and lower semisolvable.

Suppose that (ii) holds. Since in this case $B$ is a proper Hopf
subalgebra, so in particular its irreducible comodules have degree
at most $2$, we may assume inductively that $B$ is lower
semisolvable. Then $H$ is lower semisolvable as well. The
corollary is thus proven.
\end{proof}

\begin{remark} It follows from Corollary \ref{semisoluble} and \cite{MoW} that a semisimple Hopf algebra $H$
with character degrees of irreducible comodules at most $2$ have
the Frobenius property: that is, the dimensions of the irreducible
$H$-modules divide the dimension of $H$. \end{remark}

As a consequence of the above results, we have the following
proposition. This is proved by Izumi and Kosaki in the context of
Kac algebras \cite[Corollary IX.9]{ik}, and under an additional
restriction in \cite[Theorem 4.6.5]{ssld}.

\begin{proposition} Suppose $H$ is of type $(1, 2; 2, n)$ as a
coalgebra. Then $H$ is commutative. \end{proposition}

\begin{proof} If $H \neq H_{\coad}$, then $H$ fits into
an exact sequence \begin{equation}\label{udec}k \to B \to H \to k
U(\Cc) \to k,\end{equation} where $B = H_{\coad}$. We may assume
inductively that $B$ is commutative, thus the extension
\eqref{udec} is abelian. Hence \eqref{udec} induces an action of
$U(\Cc)$ on $B$ by Hopf algebra automorphisms. Moreover, since
$G(B) = G(H)$ is of order $2$, then $U(\Cc)$ acts trivially on
$G(B)$. Then $g \in Z(H)$, because $H$ is a crossed product with
respect to this action.

If $H = H_{\coad}$, then by Theorem \ref{1m2n}, we have a central
abelian exact sequence $k \to kG(H) \to H \to kF \to k$, where $F$
is a finite group. In any case we get that $G(H) \subseteq Z(H)$.
Then the proposition follows from \cite[Corollary 4.6.8]{ssld}.
\end{proof}

\subsection{Coalgebra type of a simple Hopf algebra of dimension
$60$}\label{dim60}

There are three known examples of  nontrivial semisimple Hopf
algebras which are simple as Hopf algebras in dimension $60$. Two
of them are given by the Hopf algebras $\A_0$ and $\A_1 \simeq
\A_0^*$ constructed by Nikshych \cite{nik}. We have $\A_0 =
(k\mathbb A_5)^J$, where $J \in k\mathbb A_5 \otimes k\mathbb A_5$
is an invertible twist lifted from a subgroup isomorphic to $\Z_2
\otimes \Z_2$.

A third example is given by the self-dual Hopf algebra $\B$
constructed in \cite{gn}. In this case $\B = (kD_3\otimes
kD_5)^J$, where $J$ is an invertible twist lifted from a subgroup
isomorphic to $\Z_2 \times \Z_2$.

In this subsection we show that $\A_1$ and $\B$ are the only
nontrivial simple Hopf algebras of dimension $60$ with their
respective coalgebra type. See Proposition \ref{n=1} and Corollary
\ref{b}.

\begin{remark} The coalgebra types of these examples are the following:
\begin{equation}\A_1: (1, 1; 3, 2; 4, 1; 5, 1), \quad \A_0: (1, 12;
4, 3), \quad \B: (1, 4; 2, 6; 4, 2), \end{equation} and we have
$G(\A_0) \simeq \mathbb A_4$, $G(\B) \simeq \Z_2 \times \Z_2$.

\begin{proof} The coalgebra type of $\A_1$ is that of $k^{\mathbb
A_5}$. That of $\B$ was computed in \cite[4.2]{gn}. The proof of
the statement for $\A_0 = \A_1^*$ is given during the proof of
Lemma 6.8 in \cite{bb}, using the classification of irreducible
representations in twisting of group algebras \cite{eg}.
\end{proof} \end{remark}

\medbreak In what follows $H$ will be a semisimple Hopf algebra of
dimension $60$.

\begin{proposition}\label{n=1} Suppose $|G(H)| = 1$.  Then $H \simeq k^{\mathbb A_5}$ or $H \simeq \A_1$.
\end{proposition}

\begin{proof} By \cite{ssld} every semisimple Hopf algebra of dimension $<
60$ has a nontrivial group-like element. Therefore, if $|G(H)| =
1$,  then $H$ can contain no proper Hopf subalgebra. The
proposition follows from \cite[Corollary 9.14]{ENO2}. \end{proof}

\medbreak Suppose $G(H) \neq 1$ and $H$ has an irreducible
character $\chi$ of degree $2$. Let $C \subseteq H$ be the simple
subcoalgebra containing $\chi$, and consider the Hopf subalgebra
$B = : B[\chi] = k[C\mathcal S(C)] \subseteq H_{\coad}$.

Then we have $B \simeq k^{\mathbb A_4}$, $k^{\Z_2 \times \Z_2}$ or
$k^{D_n}$, $n = 2, 3$ or $5$.

\begin{lemma}\label{n=4} Suppose $H$ is simple and $|G(H)| = 4$. Assume in addition that $H$
has irreducible characters $\chi$ and $\chi'$ of degree $2$,  such
that $G[\chi] \neq G[\chi']$. Then $H$ is isomorphic to the
self-dual Hopf algebra $\B = (kD_3\otimes D_5)^J$. \end{lemma}

\begin{proof} Since $G[\chi], G[\chi'] \neq 1$, then $G[\chi]$ and
$G[\chi']$ are distinct subgroups  of order $2$. Then $B[\chi]$
and $B[\chi']$ are distinct Hopf subalgebras of $H$ isomorphic to
$k^{D_3}$ or $k^{D_5}$.

Let $K[\chi] = k[G(H), B[\chi]]$ and $K[\chi'] = k[G(H),
B[\chi']]$. Then $K[\chi]$, $K[\chi']$ are Hopf subalgebras of
$H$, whose irreducible comodules are of dimension $1$ and $2$.
Hence $\dim K[\chi]$, $\dim K[\chi']$ equal $12$ or $20$.
Moreover, by dimension, $K[\chi]$ and $K[\chi']$ generate $H$ as
an algebra. In addition, $G(K[\chi]) = G(K[\chi']) = G(H)$ is not
cyclic.

If one of them, say $K[\chi]$, is commutative, then $G[\chi']$
which is central in $K[\chi']$, would be central in $H$, whence
$H$ would not be simple. We may assume then that $K[\chi]$ and
$K[\chi']$ are not commutative. Then there exists a twist $J\in
kG(H) \otimes kG(H)$ such that $K[\chi]^J$ and  $K[\chi']^J$ are
group algebras \cite[5.2]{ssld}. In particular, since $K[\chi]^J$
and  $K[\chi']^J$ generate $H^J$ as an algebra (because the
algebra structure is unchanged under twisting), we find that $H^J$
is cocommutative. The lemma follows from \cite[Theorem 4.10]{gn}.
\end{proof}

\begin{corollary}\label{b} Suppose $H$ is of type $(1, 4; 2, 6; 4, 2)$ as a coalgebra. If $H$
is simple, then $H$ is isomorphic to the self-dual Hopf algebra
$\B = (kD_3\otimes kD_5)^J$.
\end{corollary}

\begin{proof} Since the simple Hopf algebra $\B$ is of type $(1, 4; 2, 6; 4,
2)$ as a coalgebra, in view of Lemma \ref{n=4}, it is enough to
show that $H$ has irreducible characters  $\chi$ and $\chi'$ of
degree $2$, such that $G[\chi] \neq G[\chi']$.

In view of the coalgebra structure, there must exist irreducible
characters $\chi$ and $\chi'$ of degree $2$, such that $\chi\chi'$
is irreducible of degree $4$. Otherwise, the sum of simple
subcoalgebras of dimensions $1$ and $2$ would be a Hopf subalgebra
of $H$ of dimension $28$, which is impossible. By \cite[Theorem
2.4.2]{ssld}, we have $G[\chi] \cap G[\chi'] = 1$, thus $G[\chi]$,
$G[\chi']$ are distinct subgroups of order $2$. \end{proof}

\section{Appendix: The group $\Opext(k^N, k\Z_2)$}\label{opext}

Suppose that $F = \Z_2$ is the cyclic group of order $2$, and let
$N$ be a finite group. A matched pair $(\Z_2, N)$ necessarily has
trivial action $\vartriangleright$, whence the action
$\vartriangleleft: N \times \Z_2 \to N$ is by group automorphisms.

\medbreak Let $G = N \rtimes \Z_2$ be the corresponding semidirect
product. Note that the restriction map $H^2(N\rtimes \Z_2,
k^{\times}) \to H^2(\Z_2, k^{\times}) = 1$ is necessarily trivial.

By \cite{tahara}, there is an exact sequence
$$1 \to H^1(\Z_2, \widehat N) \to H^2(G, k^{\times})
\overset{\res}\to H^2(N, k^{\times})^{\Z_2} \overset{d_2}\to
H^2(\Z_2, \widehat N) \to \widetilde H^3(G, k^{\times}).$$

\begin{lemma}\label{z2-gral} Let $(\Z_2, N)$ be a matched pair as
above. Assume in addition that $H^2(N, k^{\times}) \simeq \Z_2$.
Then we have a group isomorphism $$\Opext(k^N, k\Z_2)/K \simeq 1,
\text{ or } \Z_2,$$ where $K \simeq d_2(H^2(N, k^{\times}))
\subseteq H^2(\Z_2, \widehat N)$.

The extensions corresponding to elements of $K$ are, as Hopf
algebras, twisting deformations of the split extension $k^N \#
k\Z_2$. \end{lemma}

\begin{proof} To prove the lemma, we shall apply the results of \cite{mastnak},
identifying $\Opext(k^N, k\Z_2)$ with the group, denoted
$H^2(k\Z_2, k^N)$ in \textit{loc. cit.}, of classes of compatible
cocycles $(\sigma, \tau)$. Recall that there are subgroups
$H^2_c(k\Z_2, k^N)$, $H^2_m(k\Z_2, k^N)$ of $\Opext(k^N, k\Z_2)$,
which are identified  with the subgroup of classes of compatible
pairs $(1, \tau)$ and $(\sigma, 1)$, respectively \cite[Section
3]{mastnak}.

In view of the assumption $H^2(N, k^{\times}) \simeq \Z_2$, we
have $H^2(N, k^{\times})^{\Z_2} = H^2(N, k^{\times})$, and
$H^1(\Z_2, H^2(N, k^{\times})) \simeq \Z_2$.

Since $k$ is algebraically closed, we have $\Opext(k^N, k\Z_2) =
H^2_c(k\Z_2, k^N)$, by \cite[Theorem 4.4]{mastnak}. Then, by
\cite[Remark pp. 418]{mastnak}, $$H^2_m(k\Z_2, k^N) = H^2_m(k\Z_2,
k^N) \cap H^2_c(k\Z_2, k^N) = d_2(H^2(N, k^{\times})),$$ where
$d_2: H^2(N, k^{\times}) \to H^2(\mathbb Z_2, \widehat N)$, is the
connecting homomorphism in \cite{tahara}.

Let $K = H^2_m(k\Z_2, k^N)$. By \cite[Theorem 7.1]{mastnak}, we
get an exact sequence
\begin{equation}\label{kac1}H^2(N, k^{\times}) \oplus K \overset{\Phi \oplus \iota}\to
\Opext(k^N, k\Z_2) \overset{\pi}\to \Z_2.
\end{equation}
With the identification $\Phi = d_2$ as in \cite{mastnak}, we see
that
$$(\Phi \oplus \iota) (H^2(N, k^{\times}) \oplus K) = K \subseteq \Opext(k^N, k\Z_2).$$
Thus, \eqref{kac1} induces an exact sequence
\begin{equation}\label{kac-seq}1 \to K \to
\Opext(k^N, k\Z_2) \overset{\pi}\to \Z_2,
\end{equation}  whence $\Opext(k^N, k\Z_2)/ K \simeq \pi(\Opext(k^N, k\Z_2)) \leq \Z_2$.

\medbreak According to \cite[Remark pp. 418]{mastnak}, the exact
sequence \eqref{kac-seq} considered in the proof of Lemma
\ref{z2-gral} restricts, in this case, to the Kac exact sequence
associated to the matched pair $(\Z_2, N)$ \cite{kac},
\cite[Theorem 7.4]{Maext}. In view of \cite{ma-ext2}, all Hopf
algebra extensions $H$ arising from elements of $K$ are, as Hopf
algebras, twisting deformations (of the comultiplication) of the
split extension. This finishes the proof of the lemma.
\end{proof}

\begin{remark}\label{split} Keep the assumptions in Lemma \ref{z2-gral}.
Since $H^2(N, k^{\times}) = \Z_2$, then either $K = 1$ or $K
\simeq \Z_2$. The possibility $K = 1$ is equivalent to $d_2 = 1$,
while $K \simeq \Z_2$ amounts to $d_2$ being injective.

Consider the restriction map $\res: H^2(N\rtimes \Z_2, k^{\times})
\to H^2(N, k^{\times})$. It follows from exactness of the sequence
in \cite{tahara}, that $K = 1$ if and only $\res$ is surjective,
and $K \simeq \Z_2$ if and only if $\res$ is trivial.

\medbreak In particular, if  $H^2(\Z_2, \widehat N) = 1$ (for
instance, if $|\widehat N|$ is odd), we have $\Opext(k^N, k\Z_2)
\simeq \pi(\Opext(k^N, k\Z_2)) \leq \Z_2$.
\end{remark}

\begin{remark} Consider the matched pair $(\Gamma, \Z_2)$ as in Section \ref{2por2},
where $\Gamma$ is a nonabelian polyhedral group. The existence of
the extensions $\mathcal A[\widetilde \Gamma]$, $\mathcal
B[\widetilde \Gamma]$, shows that the map $\pi: \Opext(k^{\Gamma},
k\Z_2) \to \Z_2$ is in this case surjective. Therefore, we have an
isomorphism $\Opext(k^{\Gamma}, k\Z_2)/K \simeq \Z_2$.
\end{remark}


\bibliographystyle{amsalpha}

\begin{thebibliography}{A}

\bibitem{ad}
\textsc{N. Andruskiewitsch} and \textsc{J. Devoto},
\emph{Extensions of Hopf algebras},
St. Petersburg Math. J.  {\bf 7} (1996), 17--52.


\bibitem{ba1}  \textsc{T. Banica}, \emph{Le groupe quantique libre $U(n)$},
Comm. Math. Phys. \textbf{190} (1997), 143--172.

\bibitem{ba2}  \textsc{T. Banica}, \emph{Symmetries of a generic coaction},
Math. Ann. \textbf{314} (1999), 763--780.

\bibitem{bb}
\textsc{T. Banica} and \textsc{J. Bichon}, \emph{Quantum groups acting on 4 points},
  J. Reine Angew. Math.  \textbf{626}  (2009), 75--114.

\bibitem{univ-cosov} \textsc{J. Bichon}, \emph{Cosovereign Hopf algebras},  J. Pure Appl. Algebra \textbf{157} (2001), 121--133.

\bibitem{binyjm}
\textsc{J. Bichon}, \emph{Cosemisimple Hopf algebras with antipode
of arbitrary finite order},  New York J. Math.  \textbf{8} (2002),
235--240.

\bibitem{B} \textsc{J. Bichon}, \emph{The representation category of the quantum group of
a non-degenerate bilinear form}, Commun. Algebra \textbf{31} (2003),
4831--4851.

\bibitem{bic3}
\textsc{J. Bichon}, \emph{Corepresentation theory of universal cosovereign Hopf algebras},  J. Lond. Math. Soc. \textbf{75}  (2007), 83--98.

\bibitem{coxeter} {\sc H. Coxeter}, \emph{Regular complex polytopes},  Cambride University Press, London (1974).

\bibitem{davydov} {\sc A. A. Davydov}, \emph{Galois algebras and monoidal functors between
categories of representations of finite groups},  J. Algebra
\textbf{244} (2001), 273--301.

\bibitem{dijkhuizen} \textsc{M. Dijkhuizen}, \emph{The double covering of the quantum group ${\rm SO}\sb q(3)$},
Rend. Circ. Mat. Palermo (2) Suppl. \textbf{37}  (1994), 47--57.

\bibitem{DL} \textsc{M. Dubois-Violette} and \textsc{G. Launer}, \emph{The quantum group of a
non-degenerate bilinear form}, Phys. Lett. B \textbf{245},
175--177 (1990).

\bibitem{eg}  {\sc P. Etingof} and {\sc S. Gelaki},
\emph{The representation theory of cotriangular semisimple Hopf
algebras}, Int. Math. Res. Not. \textbf{1999} (1999), 387–-394.

\bibitem{ENO} {\sc P. Etingof}, {\sc D. Nikshych} and {\sc V. Ostrik},
\emph{On fusion categories}, Ann. Math. (2) \textbf{162} (2005),
581--642.

\bibitem{ENO2} {\sc P. Etingof}, {\sc D. Nikshych} and {\sc V. Ostrik},
\emph{Weakly group-theoretical and solvable fusion categories},
Adv. Math. \textbf{226} (2011) 176–-205.

\bibitem{frt} \textsc{L. Faddeev}, \textsc{N. Reshetikhin} and \textsc{L. Takhtadzhyan},
\emph{Quantization of Lie groups and Lie algebras}, Leningr. Math.
J. \textbf{1} (1990) 193--225.

\bibitem{gn} \textsc{C. Galindo} and \textsc{S. Natale}, \emph{Simple Hopf algebras and deformations of finite groups},
Math. Res. Lett. \textbf{14} (2007), 943--954.

\bibitem{gelaki-nik} {\sc S. Gelaki} and {\sc D. Nikshych}, \emph{Nilpotent fusion
categories}, Adv. Math. \textbf{217}  (2008), 1053--1071.

\bibitem{ik} \textsc{M. Izumi} and \textsc{H. Kosaki},
\emph{Kac algebras arising from composition of subfactors: general
theory and classification}, Mem. Amer. Math. Soc. \textbf{158}
(2002).

\bibitem{kac}  \textsc{G.  Kac},
\textit{Extensions of groups to ring groups},  Math. USSR. Sb.
\textbf{5} (1968), 451--474.

\bibitem{kashina}  \textsc{Y. Kashina}, \emph{Classification of semisimple Hopf algebras of dimension 16},
J. Algebra \textbf{232} (2000),  617--663.

\bibitem{ksz}  \textsc{Y. Kashina}, {\sc Y. Sommerh\" auser} and {\sc Y.
Zhu}, \emph{On higher Frobenius-Schur indicators}, Mem. Amer. Math.
Soc. \textbf{855} (2006).

\bibitem{karpilovsky}
\textsc{G. Karpilovsky}, \emph{Projective representations of
finite groups}, Pure Appl. Math. \textbf{94}, Marcel Dekker, New
York (1985).

\bibitem{ks} {\sc A. Klimyk} and {\sc K. Schm\"udgen}, \emph{Quantum groups and their representations}, Texts and Monographs in Physics, Springer-Verlag, Berlin (1997).



\bibitem{larson-radford} {\sc R. Larson} and {\sc D. Radford}, \emph{Semisimple cosemisimple
Hopf algebras}, Amer. J. Math \textbf{110} (1988), 187--195.

\bibitem{li-mo}  {\sc V. Linchenko} and {\sc S. Montgomery},
\emph{A Frobenius-Schur theorem for Hopf algebras},  Algebr.
Represent. Theory \textbf{3}  (2000), 347--355.

\bibitem{mastnak}  {\sc M. Mastnak},  \emph{Hopf algebra
extensions arising from semi-direct products}, J. Algebra
\textbf{251} (2002), 413–-434.

\bibitem{ma-ff}  {\sc A. Masuoka},  \emph{Faithfully flat forms and cohomology of Hopf algebra
extensions}, Comm. Algebra \textbf{25} (1997), 1169–-1197.

\bibitem{Maext}  {\sc A. Masuoka},  \emph{Extensions of Hopf algebras},
Trab. Mat. \textbf{41/99}, FaMAF, 1999.

\bibitem{ma-ext2}  {\sc A. Masuoka},
\emph{Hopf algebra extensions and cohomology}, in: New Directions
in Hopf Algebras, MSRI Publ. \textbf{43} (2002), 167--209.

\bibitem{ma-contemp}  \textsc{A. Masuoka},  \emph{Cocycle deformations and Galois objects for some co-semisimple Hopf algebras of finite dimension},
Contemp. Math. \textbf{267} (2000), 195--214.

\bibitem{MoW}  \textsc{S. Montgomery}  and  \textsc{S. Whiterspoon},
\emph{Irreducible representations of crossed products},  J. Pure Appl. Algebra \textbf{129} (1998),
315--326.

\bibitem{movshev} {\sc M. Movshev}, \emph{Twisting in group algebras of finite groups},  Func. Anal.
Appl. {\bf 27} (1994), 240--244.

\bibitem{pqq} \textsc{S. Natale},
\emph{On semisimple Hopf algebras of dimension $pq^2$}, J. Algebra
\textbf{221} (1999),  242--278.

\bibitem{dim12} \textsc{S. Natale},
\emph{Hopf algebras of dimension 12}, Algebr. Repr, theory
\textbf{5} (2002), 445--455.

\bibitem{ssld} \textsc{S. Natale},
\emph{Semisolvability of semisimple Hopf algebras of low
dimension}, Mem. Amer. Math. Soc. \textbf{186}, no. 874, (2007).

\bibitem{extriang} \textsc{S. Natale},
\emph{Hopf algebra extensions of group algebras and
Tambara-Yamagami categories}, Algebr.  Repr. Theory \textbf{13}    (2010) 673--691.

\bibitem{NR} \textsc{W. Nichols}  and  \textsc{M. Richmond},
\emph{The Grothendieck group of a Hopf algebra}, J. Pure Appl.
Algebra \textbf{106} (1996), 297--306.

\bibitem{NZ}  {\sc W. Nichols} and {\sc M. Zoeller}, \emph{A Hopf
algebra freeness Theorem}, Amer. J.  Math. {\bf 111} (1989),
381--385.

\bibitem{nik}  {\sc D. Nikshych}, \emph{$K_0$-rings and twisting of finite-dimensional
semisimple Hopf algebras}, Comm. Algebra  {\bf 26} (1998), 321--342.


\bibitem{podles-spheres} {\sc P. Podle\' s}, \emph{Quantum spheres},
Lett. Math. Phys. \textbf{14} (1987), 193--202.

\bibitem{podles} {\sc P. Podle\' s},
\emph{Symmetries of Quantum Spaces. Subgroups and Quotient Spaces
of Quantum $\SU(2)$ and $\SO(3)$ groups}, Comm. Math. Phys.
\textbf{170} (1995), 1--20.

\bibitem{radford} {\sc D. Radford},
\emph{On an analog of Lagrange's theorem for commutative Hopf
algebras}, Proc. Amer. Math. Soc. \textbf{79} (1980), 164--166.

\bibitem{schneider} {\sc H.-J. Schneider},
\textit{Some remarks on exact sequences of quantum groups},
Comm. Algebra \textbf{21} (1993), 3337--3357.

\bibitem{stefan} {\sc D. Stefan},
\textit{Hopf algebras of low dimension}, J. Algebra \textbf{211}
(1999), 343--361.

\bibitem{stekolshchik} {\sc R. Stekolshchik}, \emph{Notes on Coxeter transformations and the McKay
correspondence}, Springer Monographs in Mathematics, Berlin
(2008).

\bibitem{suzuki}  {\sc M. Suzuki},  \emph{Group Theory, I},
Springer-Verlag, Berlin (1982).

\bibitem{tahara} {\sc K. Tahara}, \emph{On the second cohomology groups of semi-direct products},
Math. Z. \textbf{129} (1972), 365--379.

\bibitem{takeuchi-so3} {\sc M. Takeuchi}, \emph{Quantum orthogonal and symplectic grops and their embedding into quantum $\GL$},
Proc. Japan Acad. \textbf{65}, Ser. A (1989), 55--58.

\bibitem{takeuchi2} {\sc M. Takeuchi}, \emph{$q$-representations of quantum groups}, CMS Conf. Proc.
\textbf{16} (1995), 347--385.

\bibitem{wa1} \textsc{S. Wang}, \emph{Free products of compact quantum groups},
Comm. Math. Phys. \textbf{167} (1995), 671--692.

\bibitem{waterhouse} {\sc W. Waterhouse}, \emph{The module structure of certain Hopf algebra extensions},
Commun. Algebra \textbf{10} (1982), 115--120.

\bibitem{ZS} {\sc S. Zhu},  \emph{On finite dimensional Hopf algebras}, Comm. Algebra {\bf 21} (1993), 3871--3885.
\end{thebibliography}

\end{document}